\documentclass[a4paper]{article}
\usepackage{amsthm,amsfonts,amsmath,amssymb,units,bbold}
\usepackage[abbrev,nobysame]{amsrefs}
\usepackage[cp1251]{inputenc}
\usepackage[english]{babel}
\usepackage[final]{graphicx}
\usepackage{setspace}
\usepackage[12pt]{extsizes}
\begin{document}
\newtheorem{teorema}{Theorem}[section]
\newtheorem{lemma}[teorema]{Lemma}
\newtheorem{utv}[teorema]{Proposition}
\newtheorem{svoistvo}[teorema]{Property}
\newtheorem{sled}[teorema]{Corollary}
\newtheorem{con}[teorema]{Conjecture}
\newtheorem{zam}[teorema]{Remark}
\newtheorem{const}[teorema]{Construction}
\newtheorem{quest}[teorema]{Question}
\newtheorem{problem}[teorema]{Problem}

\author{A. A. Taranenko\thanks{Sobolev Institute of Mathematics, Novosibirsk, Russia. 
 \texttt{taa@math.nsc.ru}}}
\title{On the K\"onig--Hall--Egerv\'ary theorem for multidimensional matrices and multipartite hypergraphs}
\date{December 15, 2020}

\maketitle

\begin{abstract}
One of possible interpretations of the well-known K\"onig--Hall--Egerv\'ary theorem is a full characterization of all bipartite graphs extremal for fractional matchings of a given weight (or, equivalently, a characterization of $(0,1)$-matrices extremal for partial fractional diagonals of a given length). In this paper we initiate the study of $d$-partite $d$-uniform hypergraphs  that are extremal for fractional perfect matchings (or, equivalently, $d$-dimensional $(0,1)$-matrices that are extremal for polydiagonals). For this purpose, we analyze similarities and differences between $2$-dimensional and multidimensional cases and put forward a series of questions and conjectures on properties of multidimensional extremal matrices (extremal hypergraphs).  We also prove these conjectures for several parameters and provide a number of supporting constructions and examples.
\end{abstract}

\section*{Introduction}

The K\"onig--Hall--Egerv\'ary theorem is one of the fundamental results in discrete mathematics.

\begin{teorema}[K\"onig--Hall--Egerv\'ary]
Let $A$ be a $(0,1)$-matrix of order $n$. The minimum number of lines (rows and columns) sufficient for covering all unity entries of $A$ is equal to the maximum number of ones that do not share the same rows and columns. 

or

A $(0,1)$-matrix $A$ of order $n$ contains a unity partial diagonal of length $l$ if and only if there are no $s \times t$ zero submatrices in $A$ for which $s+t = 2n - l + 1.$

or

For a bipartite graph, the minimum size of a vertex covering is equal to the maximum size of a matching.
\end{teorema}

A variety of proof strategies are known for this theorem. One of them (see, for example, book~\cite{schrijver.combopt} by Schrijver) is based on linear programming, fractional matchings, and fractional vertex covers.   The K\"onig--Hall--Egerv\'ary theorem may be also expressed as a complete characterization of matrices (or graphs) extremal in some sense.  

To illustrate this, we define a \textit{polyplex} $K$ of order $n$ and of weight $W$ to be a nonnegative matrix order $n$ in which the sum of entries over each row and each column is not greater than $1$ and the sum of all entries is equal to $W$. In other words, a polyplex  is a fractional matching in a balanced bipartite graph. 
We call a polyplex of the maximum possible weight (i.e. weight $n$) to be a \textit{polydiagonal}. 

We will say that a $(0,1)$-matrix $A$ is extremal for polyplexes of weight $W$ if there are no polyplexes $K$ of weight $W$ or greater with $supp (K) \subseteq supp (A)$ but after adding any entry to the support $A$ there is a polyplex of weight at least $W$ within the support of the resulting matrix.  
Then one more equivalent form of the K\"onig--Hall--Egerv\'ary theorem is the following: all $(0,1)$-matrices extremal for polyplexes of weight $W$ are exactly the matrices in which zero entries form an $s \times t$ submatrix with $s + t = 2n - \left\lceil W \right\rceil + 1$.

For multidimensional matrices, we introduce a  polyplex so that it is equivalent to a fractional matching in a $d$-partite $d$-uniform hypergraph with parts of equal sizes (such hypergraphs are often called balanced).
The main aim of the present paper is, using a geometric interpretation and theory of linear programming, to take several steps towards the generalization of K\"onig--Hall--Egerv\'ary theorem for a multidimensional case and investigate multidimensional matrices extremal for polyplexes. In other words, we research balanced $d$-partite $d$-uniform hypergraphs extremal for fractional matchings of a given weight.

There are plenty of results on fractional matchings concerning sufficient conditions for a hypergraph to contain a fractional matching of a given weight.  Some of them look for the minimum degree for the existence of perfect matchings. For example, with the help of a geometric approach Keevash and Mycroft~\cite{KeevMyc.geommatch} obtain asymptotic degree conditions for the existence of perfect matchings that generalizes many other results. Moreover, in this book it can be found a number of references on papers concerning minimum degree problems for matchings in hypergraphs. For the survey on these and related problems, see also~\cite{RodlRus.hyperDirac} by R\"odl and Ruci\'nski.

For a balanced $d$-partite $d$-uniform hypergraph, the degree condition sufficient for the existence of a perfect matching was proved in~\cite{AhaGeoSpr.pminhyper}.
The fractional matching polytope of a hypergraph is studied, for instance, in~\cite{FurKahnSeym.fracmatch}. At last, the recent progress on the connection between the number of edges and the matching number of a hypergraph can be found in~\cite{HuLohSud.matchhyp}.

One more possible generalization of the K\"onig--Hall--Egerv\'ary theorem is the following conjecture often attributed to Ryser (see~\cite{BestWan.Rysconj} for more on the history of this conjecture): in a $d$-partite $d$-uniform hypergraph the size of the minimum vertex cover is not greater than the size of the maximum matching multiplied by $d-1$. The best progress is that this conjecture is true for $3$-partite hypergraphs~\cite{aharoni.rysconj}; extremal $3$-partite hypergraphs for the conjecture were described in~\cite{HaxNarSza.Ryserconj}. For $d = 4$ and $d=5$, an approximate version of the conjecture was proved by Haxell and Scott~\cite{HaxScott.Rysconj}.

Finally, there are generalizations of the Hall theorem of distinct representatives in literature. One of them is stated for hypergraphs by Aharoni and Haxell~\cite{AhaHax.Hallth}, and a vector version of the Hall theorem for common transversals is proved by Woodall~\cite{woodall.vecttrans}.

Let us describe the structure of this paper. 
In Section 1 we introduce the main concepts including definitions on not only multidimensional polyplexes and extremal matrices but (fractional) hyperplane covers.  When a polyplex corresponds to a fractional matching in a $d$-partite hypergraph, a hyperplane cover of a multidimensional matrix is exactly a fractional vertex cover of the same hypergraph. 

In Section 2 we recall a duality between problems on the maximum fractional matching and the minimum fractional vertex hyperplane cover in a hypergraph and state the same duality for polyplexes and hyperplane covers. Also, we deduce some easy consequences of the linear programming theory applied to polyplexes and hyperplane covers.

In Section 3 we probe our approach on the case of $2$-dimensional matrices providing one of the proofs of the K\"onig--Hall--Egerv\'ary theorem.  We also find the differences between the $2$-dimensional and multidimensional cases and reveal some properties of $2$-dimensional extremal matrices. In the second part of this section, we formulate a research programme for multidimensional extremal matrices, raise the main questions, and propose several conjectures. In particular, one of these conjectures is equivalent to that every $d$-partite hypergraph extremal for a fractional perfect matching is extremal for an integer perfect matching too.  If we succeed in completing the above programme, we obtain a characterization of extremal matrices for polyplexes that is equivalent to the multidimensional K\"onig--Hall--Egerv\'ary theorem. 

Following our research programme, in Section 4 we consider a question on the equivalence between extremal matrices and their optimal hyperplane covers and in Section 5 we provide some properties of optimal hyperplane covers. A series of constructions of extremal matrices is given in Section 6, and Section~7 is devoted to the property of diagonal extremality.  In Section 8 we look closer at extremal matrices of orders $2$ and $3$ and prove some additional properties for them. The paper concludes with a list of examples of extremal matrices of small orders and dimensions.

The main results of the present consideration are characterizations of all extremal matrices of big deficiencies and extremal matrices with two-value optimal hyperplane covers. In particular, all main conjectures from Section 3 are confirmed for such classes of matrices.

\section{Definitions}

A \textit{$d$-dimensional matrix $A$ of order $n$} is an array $(a_\alpha)_{\alpha \in I^d_n}$, $a_\alpha \in\mathbb R$, with the set of indices $I_n^d= \left\{ (\alpha_1, \ldots , \alpha_d):\alpha_i \in \left\{1,\ldots,n \right\}\right\}$. For $x \in \mathbb{R}$, we will say that $a_{\alpha}$ is an \textit{$x$-entry} of a matrix $A$ if $a_{\alpha} = x$.
A multidimensional matrix $A$ is called \textit{nonnegative} if all entries of $A$ are nonnegative.  A \textit{$(0,1)$-matrix} is a multidimensional matrix whose all entries are equal to zero or one. The \textit{support} $supp(A)$ of a nonnegative matrix $A$ is the set of indices $\alpha$ of all nonzero entries $a_\alpha$.

 A \textit{hyperplane} in a $d$-dimensional matrix $A$ is a $(d-1)$-dimensional submatrix of $A$ obtained by fixing one of coordinates. Denote by $\Gamma_{i,j}$ the $j$-th hyperplane of direction $i$ in the matrix $A$ (submatrix in which the $i$-th coordinate is assigned to be $j$). Given a $d$-dimensional matrix $A$ of order $n$, let $A_\alpha$ be the $d$-dimensional submatrix of order $n-1$ obtained by deleting all hyperplanes containing index $\alpha$. 

Multidimensional matrices $A$ and $A'$ are said to be \textit{equivalent} if one can be turned to the other by permutations of hyperplanes of each direction and by permutations of directions (matrix transpositions).

A \textit{polyplex} $K$ of \textit{weight} $W$ is a nonnegative multidimensional matrix in which the sum of entries over each hyperplane is not greater than $1$ and the sum of all entries equals $W$. Note that the set of all polyplexes of the same weight is a convex polytope.   

A polyplex of weight $n$ and order $n$ is said to be a \textit{polydiagonal}.
The simplest example of a polydiagonal is a \textit{diagonal} that is a $(0,1)$-matrix with exactly one unity entry in each hyperplane. 

Strictly speaking,  it would be more correct to use terms like ``partial fractional diagonal'' instead of ``polyplex''  and ``fractional diagonal'' instead of ``polydiagonal''. But we will use the present notions for the sake of brevity and convenience.  Plexes appear in studies of latin squares and latin hypercubes where they mean sets of entries evenly distributed among all hyperplanes and symbols of a latin hypercube.  See survey~\cite{wanless.surv} for results on plexes in latin squares and paper~\cite{my.iter} for a generalization of plexes for multidimensional matrices. 

We will say that  a multidimensional \textit{$(0,1)$-matrix $A$ contains a polyplex $K$} (or that \textit{$K$ is polyplex in $A$}) if $A$ and $K$ have the same order and dimension and  $supp (K)\subseteq supp(A)$.  An \textit{optimal} polyplex in a matrix $A$ is a polyplex of the maximum weight contained in the matrix $A$.

A \textit{hyperplane cover} of a $d$-dimensional $(0,1)$-matrix $A$ of order $n$ is a  $(d \times n)$-table $\Lambda = (\lambda_{i,j})$ assigning  nonnegative weights to all hyperplanes of $A$ in such a way so each unity entry of $A$ \textit{is covered with weight} not less than 1, i.e. we demand that  for each $\alpha \in supp (A)$ it holds $\sum\limits_{\Gamma_{i,j} \ni \alpha}  \lambda_{i,j} \geq 1.$ 

The \textit{weight} of a hyperplane cover $\Lambda$ is the sum of all its entries $\lambda_{i,j}$.
It is not hard to see that the set of all hyperplane covers of a given weight is a convex polytope.

We call a hyperplane cover $\Lambda$ of a $(0,1)$-matrix $A$ \textit{optimal} if it has the minimum weight among all hyperplane covers of $A$. Note that every $(0,1)$-matrix $A$ of order $n$ has a hyperplane cover of weight $n$ because we can always cover the support of  $A$ by $n$ hyperplanes of the same direction having weight 1. Therefore, the weight of an optimal hyperplane cover of any matrix of order $n$ is not greater than $n$.

Given a nonnegative table $\Lambda$, we define a $(0,1)$-matrix $A(\Lambda)$ so that the support of $A(\Lambda)$ is the set of indices covered with weight at least $1$ by $\Lambda$. Also for a given index $\alpha$ let us denote by $\Lambda_\alpha$ the $d \times (n-1)$-table obtained from $\Lambda$ by deleting all weights of  hyperplanes containing the index $\alpha$.

A multidimensional  $(0,1)$-matrix $A$ of order $n$ is called \textit{stepped} if for every direction $i$ of hyperplanes and for all $1 \leq j_1 < j_2 \leq n$ the support of the hyperplane $\Gamma_{i,j_1}$ includes the support of the hyperplane $ \Gamma_{i,j_2}$.
Note that for every hyperplane cover $\Lambda$ the matrix $A(\Lambda)$ is equivalent to some stepped matrix $B$ since at each direction we can arrange the weights of hyperplanes  in descending order.  On the other hand,  not for every stepped matrix $B$ there is a hyperplane cover $\Lambda$ such that $B = A(\Lambda)$.
Multidimensional stepped matrices of order $2$ are known as monotonic boolean functions;  $(0,1)$-matrices in echelon form serve as an example of $2$-dimensional stepped matrices.

A $d$-dimensional  $(0,1)$-matrix $A$ of order $n$ is called \textit{extremal} if $A$ contains no polydiagonals and after adding any entry to the support of $A$ the resulting matrix contains a polydiagonal. The \textit{deficiency} $\delta$ of an extremal matrix $A$ of order $n$ is the difference between $n$ and the weight of the optimal polyplex in $A$.

At last we introduce a special class of extremal matrices. We will say that a multidimensional $(0,1)$-matrix $A$ is \textit{diagonally extremal} if it contains no polydiagonals and for each index $\alpha$ such that $a_\alpha = 0$ the submatrix $A_{\alpha}$ contains a diagonal.

\section{Linear programming approach}

The problem on the maximum weight of a polyplex $K$ contained in a given $d$-di\-men\-sio\-nal $(0,1)$-matrix $A$ of order $n$ can be stated as the following linear programming problem: 
\begin{gather} \label{maxplexpr}
\sum\limits_{\alpha} k_\alpha \rightarrow \max; \notag\\
\sum \limits_{\alpha \in \Gamma_{i,j}}  k_{\alpha} \leq 1 \mbox{ for all } i = 1, \ldots, d, ~ j = 1, \ldots, n; \\
k_{\alpha} \geq 0  \mbox{ for all } \alpha \in supp (A); ~~~ k_{\alpha} = 0  \mbox{ for all } \alpha \not\in supp (A).
\notag
\end{gather}

Since all constraints use only integers, each multidimensional $(0,1)$-matrix has an optimal polyplex with rational entries.  It is well known that there exists a polynomial algorithm depending on the number of variables and on the size of constraint parameters to solve a linear programming problem. Thus, an optimal polyplex for any multidimensional matrix can be found in polynomial time.

For a given $d$-dimensional $(0,1)$-matrix $A$ of order $n$, a problem on an optimal  hyperplane cover $\Lambda$ looks as follows:
\begin{gather} \label{mincoverpr}
\sum\limits_{i=1}^d \sum\limits_{j=1}^n \lambda_{i,j} \rightarrow \min; \notag\\
\sum \limits_{\Gamma_{i,j} \ni \alpha}  \lambda_{i,j} \geq 1 \mbox{ for all } \alpha \in supp (A); \\
\lambda_{i,j} \geq 0 \mbox{ for all } i = 1, \ldots, d, ~ j = 1, \ldots, n.
\notag
\end{gather}  
Similarly to polyplexes,  for each multidimensional matrix $A$ there exists an optimal hyperplane cover $\Lambda$ having all rational entries and it can be found in polynomial time.

Problem~(1) is equivalent to a problem on the maximum weight of a fractional perfect matching in a balanced $d$-partite $d$-uniform hypergraph, and Problem~(2) aims for a vertex cover of minimum weight in such hypergraphs. The constraint matrices $U$ and $U^{T}$ for both problems are submatrices of the vertex-hyperedge incidence matrix of the complete $d$-partite $d$-uniform hypergraph with parts of size $n$.

As for fractional matchings and fractional vertex covers in general hypergraphs, we have the following.

\begin{teorema} \label{multikonig}
Given a $d$-dimensional $(0,1)$-matrix $A$ of order $n$, the maximum weight of a polyplex contained in $A$ is equal to the minimum weight of its hyperplane cover.
\end{teorema}

\begin{proof}
Problem~(2) is dual to Problem~(1). The duality theorem for linear programming implies that the optimal value for the objective function of Problem~(1) coincides with the optimal value in Problem~(2).
\end{proof}

We finish this section with one more easy corollary of linear programming.

\begin{teorema} \label{slackuse}
Assume that $A$ is a $d$-dimensional $(0,1)$-matrix of order $n$, $\Lambda$ is its optimal hyperplane cover, and $K$ is an optimal polyplex  in $A$.
\begin{enumerate}
\item If entry $k_{\alpha} > 0$ then index $\alpha$ is covered with weight $1$ by $\Lambda$; if some index $\alpha$ is covered by $\Lambda$ with not unity weight then $k_\alpha = 0$ for all optimal polyplexes in $A$.
\item If $\lambda_{i,j} > 0$ then the sum of entries of $K$ in the hyperplane $\Gamma_{i,j}$ equals $1$; if in a hyperplane $\Gamma_{i,j}$ the sum of entries of $K$ is not equal to $1$ then $\lambda_{i,j} = 0$ for all optimal hyperplane covers of $A$.
\end{enumerate} 
\end{teorema}

\begin{proof}
By the complementary slackness theorem applied  Problems~(1) and~(2), we have that for each optimal polyplex $K$ in $A$ and for each optimal hyperplane cover $\Lambda$ it holds
$$k_{\alpha} \left( 1 -\sum\limits_{i=1}^d \lambda_{i,\alpha_i} \right) = 0 \mbox{ for each index } \alpha \in supp(A),$$
and
$$\lambda_{i,j} \left( 1- \sum\limits_{\alpha \in \Gamma_{i,j}} k_{\alpha}\right) = 0 \mbox{ for each hyperplane } \Gamma_{i,j}.$$
The theorem immediately follows from these equalities.
\end{proof}

\section{$2$-dimensional case, main questions and conjectures} \label{problemsection}

We start this section with reminding one of the proofs of the K\"onig--Hall--Egerv\'ary theorem. For this purpose we need one more definition.
 
A rectangular integer matrix is called \textit{totally unimodular} if the determinant of each its square submatrix is equal to $0$ or $\pm 1$.

The proof of the following proposition can be found in book~\cite{schrijv.thLP}.

\begin{utv} \label{unimod}
Let $Q$ be a rectangular $(0,1)$-matrix of size $n \times l$. Assume that columns of $Q$ can be divided into two groups so that rows of the resulting submatrices contain no more than one unity entry. Then the matrix $Q$ is totally unimodular. 
\end{utv}

We are ready now to provide the proof for one of equivalent forms of the K\"onig--Hall--Egerv\'ary theorem.

\begin{teorema} \label{2dimkonig}
Given a $2$-dimensional $(0,1)$-matrix $A$ of order $n$, the maximum length of a positive partial diagonal in $A$ is equal to the minimum number of lines sufficient for covering all unity entries.
\end{teorema}

\begin{proof}
Note that the problem on the maximum length of a positive partial diagonal in $A$ is the integer version of Problem~(1). Similarly, the problem on the minimum number of lines covering all unity entries of $A$ is integer Problem~(2). 

In the $2$-dimensional case, the constraint matrix $U$ for one of these problems is an incidence matrix of a bipartite graph. By Proposition~\ref{unimod}, the matrix $U$ is totally unimodular. It is known (see, for example,~\cite{schrijv.thLP}) that if $U$ is a totally unimodular matrix then for all integer vectors $b$ the polyhedron $\left\{x: Ux \geq b \right\}$ has only integer vertices. 
Therefore, both integer Problems~(1) and~(2) have the same integer optimal value.
\end{proof}

Let us consider $2$-dimensional extremal matrices in more details. 
Theorem~\ref{2dimkonig} means that all $2$-dimensional extremal matrices of order $n$ are exactly the $(0,1)$-matrices with $k$ rows and $n-k-1$ columns filled by $1$ and with all other entries equal zero.
Also, from the proof of Theorem~\ref{2dimkonig} we have that any $2$-dimensional matrix $A$ contains an integer optimal polyplex  (i.e., a partial positive diagonal of the maximal length) and any optimal polyplex in $A$ is a concave sum of partial positive diagonals. All $2$-dimensional extremal matrices have deficiency $\delta = 1$ because  the weight of their optimal polyplex  is equal to $n-1$. 

The optimal hyperplane cover $\Lambda$ of  a $2$-dimensional extremal matrix is unique and is a $(0,1)$-table with $2$ rows and $n$ columns such that one row of $\Lambda$ contains $k$ ones and the other row has $n-k-1$ ones.

It also can be checked that an entry of a $2$-dimensional extremal matrix is contained in some optimal polyplex if and only if it is covered by weight exactly $1$ in the optimal hyperplane cover $\Lambda$.

At last, if we replace any zero entry of a $2$-dimensional extremal matrix by a one then the resulting matrix contains a diagonal. It means that every $2$-dimensional extremal matrix is diagonally extremal.   

Multidimensional extremal matrices are much more complicated objects than the $2$-dimensional case. Let us set up the principal questions studied in this paper.

\bigskip
\begin{quest}
\begin{center}
Does there exist a $1$-to-$1$ correspondence between extremal multidimensional matrices and their optimal hyperplane covers? 
\end{center}
\end{quest}

We believe that the answer to this question is positive. We have no examples of extremal matrices at which this correspondence is failed and we prove it for some classes of extremal matrices.   In Section~4, we consider this question in more details. Based on our studies, we propose the following conjecture.

\begin{con} \label{uniquecon}
Every extremal matrix has a unique optimal hyperplane cover.
\end{con}

\bigskip
\begin{quest}
\begin{center}
What are the possible values of deficiencies $\delta$ of extremal matrices? Given $n$ and $d$, how small can the deficiency of a $d$-dimensional extremal matrix of order $n$ be?
\end{center}
\end{quest}

It is easy to state that the deficiency of multidimensional extremal matrices is between zero and one.

\begin{utv} \label{weightextr}
If $\delta$ is a deficiency of an extremal matrix $A$,  then $0 < \delta \leq 1$.
\end{utv}
\begin{proof}
Since $A$ contains no polydiagonals, we have $\delta > 0$. 

By the definition, if we add some entry to the support of $A$, then we obtain a matrix with a polydiagonal $K$. Since $k_\alpha \leq 1$, the weight of the polyplex $K$ constricted back to the matrix $A$ is not less than $n-1$. Therefore, the deficiency $\delta$ of $A$ is at most $1$.
\end{proof}

As we noted before, the weight of an optimal polyplex is rational, so the deficiency $\delta$ of any extremal matrix is some rational number $\nicefrac{p}{q}$. The denominator $q$ of $\delta$ is bound in modulus by the largest minor of the incidence matrix of a $d$-partite hypergraph, so it does not exceed $d^{nd}$.

Considered examples of extremal matrices and the results of our investigation motivate us to propose the following.

\begin{con} \label{defcon}
The deficiency $\delta$ of any extremal matrix is equal to $\nicefrac{1}{m}$ for some $m \in \mathbb{N}$.
\end{con}

\bigskip
\begin{quest} \label{charcoverquest}
\begin{center}
Can we describe and enumerate all optimal hyperplane covers of extremal matrices (and extremal matrices itself)?
\end{center}
\end{quest}

In the next sections, we provide some necessary conditions for a hyperplane cover to define an extremal matrix and propose constructions of extremal matrices and their optimal hyperplane covers. Also, we give a characterization of all extremal matrices of big deficiencies. Our research leads us to the following conjecture.

\begin{con}\label{multicon}
If $\Lambda$ is an optimal  hyperplane cover of an extremal matrix of deficiency $\delta$ then all entries $\lambda_{i,j}$ are integer multiples of $\delta$.
\end{con}

Since a convex linear combination of optimal hyperplane covers of some matrix is an optimal hyperplane cover, we note that Conjecture~\ref{multicon} implies Conjecture~\ref{uniquecon} and if for some extremal matrix Conjecture~\ref{multicon} is true then such matrix has a unique optimal hyperplane cover.

\bigskip
\begin{quest}
\begin{center}
What are the properties of multidimensional extremal matrices?
\end{center}
\end{quest}

We know that all $2$-dimensional matrices are diagonally extremal. We do not find examples of multidimensional extremal matrices that are not diagonally extremal, and all our constructions preserve this property. So we propose the following.

\begin{con} \label{diagextrcon}
Every multidimensional extremal matrix is diagonally extremal.
\end{con}

Note that in terms of hypergraphs this conjecture means that every extremal for fractional perfect matchings balanced $d$-partite $d$-uniform hypergraph is also extremal for integer perfect matchings.

\section{Equivalence of extremal matrices and their optimal hyperplane covers}

This section aims to set up a natural bijection between extremal matrices and their optimal hyperplane covers. 
First of all, let us prove that the support of any extremal matrix is exactly the set of indices covered with weight at least $1$ by optimal hyperplane covers.

\begin{teorema} \label{extrcode}
If $\Lambda$ is an optimal hyperplane cover of an extremal matrix $A$ then $A= A(\Lambda)$.
\end{teorema}

\begin{proof}
By the definition of the optimal hyperplane cover, $A \subseteq A(\Lambda)$.

Since $\Lambda$ is an optimal hyperplane cover of an extremal matrix, the weight of $\Lambda$ is less than $n$, and $A(\Lambda)$ does not contain a polydiagonal. Thus, if the support $A$ is strictly contained in the support $A(\Lambda)$, then we have a contradiction with the extremality of the matrix $A$. 
\end{proof}

As an immediate corollary of this theorem, we have that every extremal matrix is equivalent to some stepped matrix.

Our next aim is to provide sufficient conditions for the existence of a unique optimal hyperplane cover and state this uniqueness for several classes of extremal matrices.  Note that for other optimal objects a similar statement often does not hold.  For example, even $2$-dimensional extremal matrices may contain many different optimal polyplexes. Moreover, there is plenty of non-extremal matrices $A$ having not unique optimal hyperplane cover.

By the definition, an optimal hyperplane cover of a matrix $A$ is unique if and only if Problem~(\ref{mincoverpr}) for the matrix $A$ has a unique solution. The uniqueness of optimal solutions in linear programming was studied in many papers. In our work  we are guided by~\cite{FloPar.encycopt} and~\cite{mangas.LPuniq}.

It is well known that the dual linear programming problem has a unique optimal solution if there exists a nondegenerate optimal solution of the primal problem. Equivalently, it means that the primal problem has an optimal solution at which the linear independence constraint qualification (LICQ) is fulfilled. In terms of optimal polyplexes and optimal hyperplane covers this condition has the following form.

\begin{utv} \label{optuniind}
Let $A$ be a multidimensional $(0,1)$-matrix and let $K$ be an optimal polyplex in $A$.  For each hyperplane $\Gamma_{i,j}$ assign a $(0,1)$-vector $v_{i,j}$ of length $|supp(A)|$ such that $v_{i,j}(\alpha) = 1$ if and only if $\alpha \in \Gamma_{i,j} \cap supp(A)$.  Suppose $V$ is the set of all vectors $v_{i,j}$ such that $\sum\limits_{\alpha \in \Gamma_{i,j} } k_{\alpha} =1$.  If $V$ is a set of  linearly independent vectors, then the matrix $A$ has a unique optimal hyperplane cover.
\end{utv}

Another sufficient condition for the uniqueness of the optimal hyperplane cover is that every optimal hyperplane cover $\Lambda$ of a matrix $A$ covers with weight $1$ the same set of indices.

\begin{utv} \label{optuniexact}
Let $A$ be a multidimensional $(0,1)$-matrix.
Suppose that all $\alpha \in supp (A)$ that does not belong to the support of any optimal polyplex in $A$ are covered with weight greater than $1$ by any optimal hyperplane cover. Then  $A$ has a unique optimal hyperplane cover.
\end{utv}

To formulate an alternative version of this condition we introduce one more definition. Given a hyperplane cover $\Lambda$, we will say that $\alpha$ is an \textit{upper index} for $\Lambda$ if $\alpha$ is covered with weight at least $1$ by $\Lambda$ and all indices $\beta$  for which all $\lambda_{i,\beta_i} \leq \lambda_{i,\alpha_i}$ and $\sum\limits_{i=1}^d \lambda_{i,\beta_i} < \sum\limits_{i=1}^d \lambda_{i,\alpha_i}$ are covered by $\Lambda$ with weight less than $1$.

\begin{utv} \label{optuniupper}
If an optimal hyperplane cover $\Lambda$ of a matrix $A$ covers all upper indices with weight $1$, then $\Lambda$ is the unique optimal hyperplane cover for this matrix.
\end{utv}

Let us give the last sufficient condition on the uniqueness of optimal hyperplane cover that is derived from~\cite{mangas.LPuniq}.

\begin{utv} \label{optunismall}
A hyperplane cover $\Lambda$ is the unique optimal hyperplane cover of a $(0,1)$-matrix $A$ if it remains an optimal solution for Problem~$(2)$ with objective function $\sum\limits_{i=1}^d \sum\limits_{j=1}^n (1 \pm \varepsilon_{i,j}) \lambda_{i,j}$ for all arbitrary but sufficiently small $\varepsilon_{i,j}$.
\end{utv}

We propose the following conjectures that imply the uniqueness of optimal hyperplane covers of extremal matrices.

\begin{con} [equivalent to Conjecture~\ref{uniquecon}] \label{hypuniind}
Let $A$ be a multidimensional extremal matrix and let $K$ be an optimal polyplex in $A$.  Then the set of $(0,1)$-vector $v_{i,j}$ such that $v_{i,j}(\alpha) = 1$ if and only if $\alpha \in \Gamma_{i,j} \cap supp(A)$ and $\sum\limits_{\alpha \in \Gamma_{i,j} } k_{\alpha} =1$ is linearly independent.
\end{con}

\begin{con} [equivalent to Conjecture~\ref{uniquecon}] \label{hypuniexact}
Let $\Lambda$ be an optimal hyperplane cover of an extremal matrix $A$. If index $\alpha$ is covered with weight $1$ by $\Lambda$ then there exists an optimal polyplex in $A$ whose support contains $\alpha$.
\end{con}

We conclude this section with a list of cases when extremal matrices have unique optimal hyperplane covers. All proofs will be given in future sections.
\begin{utv} [supporting Conjecture~\ref{uniquecon}]
If $A$ is one of the following extremal matrices then it has a unique optimal hyperplane cover.
\begin{enumerate}
\item $A$ is an extremal matrix of deficiency $1$, $\nicefrac{1}{2}$ or $\nicefrac{1}{3}$.
\item $A$ has an optimal hyperplane cover $\Lambda$ with all entries $\lambda_{i,j} \in \{0, \lambda\}$ for some $\lambda$.
\end{enumerate}
\end{utv}

\section{Conditions on optimal hyperplane covers of extremal matrices} \label{condsection}

The main aim of this section is to reveal some necessary and sufficient conditions on hyperplane covers $\Lambda$ for the matrix $A(\Lambda)$ to be extremal.
We start with conditions following from Theorem~\ref{slackuse} (or from the complementary slackness theorem of linear programming).

\begin{utv} \label{zeroexistextr}
If $\Lambda$ is an optimal hyperplane cover of an extremal matrix $A$, then each row of $\Lambda$ contains at least one zero entry. 
\end{utv}

\begin{proof}
Since $A$ has no polydiagonals, for each direction $i$ there exists a hyperplane $\Gamma_{i,j}$ such that the sum of entries of some optimal polyplex over the hyperplane $\Gamma_{i,j}$ is less than $1$. Theorem~\ref{slackuse} implies that $\lambda_{i,j} = 0.$
\end{proof}

\begin{utv} \label{weightone}
If $\Lambda$ is an optimal hyperplane cover of an extremal matrix $A$ with deficiency $\delta < 1$, then in each hyperplane $\Gamma_{i,j}$ there exists index $\alpha$ covered by $\Lambda$ with weight $1$. 
\end{utv}

\begin{proof}
Since $\delta < 1$, every hyperplane $\Gamma_{i,j}$ of $A$ contains a positive entry $k_{\alpha}$ of some  optimal polyplex $K$.  By Theorem~\ref{slackuse}, we have that index $\alpha$ is covered with weight $1$ by the hyperplane cover $\Lambda$.
\end{proof}

Let us consider other necessary conditions on optimal hyperplane covers of extremal matrices.
From the definitions we deduce that $\Lambda$ is an optimal hyperplane cover of $A(\Lambda)$ and, moreover, $A(\Lambda)$ is an extremal matrix of order $n$ if and only if both of the following demands hold:
\begin{enumerate}
\item the matrix $A(\Lambda)$ contains a polyplex of the same weight as $\Lambda$;
\item there are no hyperplane covers $\Lambda'$ of weight less than $n$ such that $A(\Lambda) $ is contained in $A(\Lambda')$.
\end{enumerate}

The first demand is equivalent to that a certain system of linear equations is solvable and has a nonnegative solution. Indeed, if $\Lambda$ is an optimal hyperplane cover of weight $W$ for an extremal matrix $A$, then  Problem~(\ref{maxplexpr}) turns to be equivalent to the following system of equations:
\begin{gather*}
\sum\limits_{\alpha \in \Gamma_{i,j}} k_{\alpha} = 1 \mbox{ for all } \lambda_{i,j} >0; \\
\sum\limits_{\alpha \in supp(A)} k_{\alpha} = W.
\end{gather*}

\begin{problem} [related to Question~\ref{charcoverquest}]
What conditions on the table $\Lambda$ guarantee that it defines a multidimensional matrix with an optimal polyplex of the same weight? In other words, for which nonnegative tables $\Lambda$ the above system has a nonnegative solution?
\end{problem}

Let us now consider the second demand claiming that there are no hyperplane covers of nonmaximal weight defining a matrix with bigger support. With the help of this demand, we derive several necessary conditions on optimal hyperplane covers of extremal matrices.

\begin{utv} \label{extrnescond}
If $\Lambda$ is an optimal hyperplane cover of an extremal matrix $A$ of deficiency $\delta$, then there are no indices $\alpha$ covered by $\Lambda$ with weight strictly between $1 - \delta$ and $1$.
\end{utv}

\begin{proof}
Assume that index $\alpha \not\in supp (A)$ is covered by $\Lambda$ with weight $1 - \varepsilon$ for $0 < \varepsilon < \delta$. Consider the hyperplane cover $\Lambda'$ that coincides with $\Lambda$ in all entries except for $\lambda_{1, \alpha_1}' = \lambda_{1, \alpha_1} + \varepsilon$. The weight of $\Lambda'$ is less than $n$ and the matrix $A$ is strictly contained in $A(\Lambda')$ that contradicts to the extremality of $A$. 
\end{proof}

\begin{utv} \label{boundsinrow}
If $\Lambda$ is an optimal hyperplane cover of an extremal matrix $A$ of deficiency $\delta$, then the difference between two non-equal entries $\lambda_{i,j}$ and $\lambda_{i,l}$ from the same row of $\Lambda$  is not less than $\delta$. 
\end{utv}

\begin{proof}
For the sake of contradiction, assume that for some $i, j$ and $l$ we have $\lambda_{i,j} = \lambda_{i,l} + \varepsilon$, where $0 < \varepsilon < \delta$. Since $A = A(\Lambda)$ and $\Lambda$ has the minimum weight among all hyperplane covers of $A$, the support of the hyperplane $\Gamma_{i,l}$ is strictly contained in the support of the hyperplane $\Gamma_{i,j}$. 

Consider the hyperplane cover $\Lambda'$ that coincides with $\Lambda$ in all entries except for $\lambda_{i,l}' = \lambda_{i, l} + \varepsilon$, that is $A(\Lambda')$ differs from $A$ in that the hyperplane $\Gamma_{i,l}$ is replaced by the hyperplane $\Gamma_{i,j}$. The weight of $\Lambda'$ is less than $n$ and the matrix $A$ is strictly contained in $A(\Lambda')$, so we have a contradiction to the extremality of $A$. 
\end{proof}

\begin{utv} \label{edgebound}
If $\Lambda$ is an optimal hyperplane cover of an extremal matrix $A$ of deficiency $\delta$, then for every $\lambda_{i,j}$ not equal to zero or one  we have $\delta \leq \lambda_{i,j} \leq 1 - \delta$.
\end{utv}

\begin{proof}
By Proposition~\ref{zeroexistextr}, each row $i$ of $\Lambda$ contains at least one zero entry, so by Proposition~\ref{boundsinrow} all nonzero $\lambda_{i,j}$ are not less than $\delta$. On the other hand, if there exists $\lambda_{i,j}$ such that $1 - \delta < \lambda_{i,j} < 1$,  then by Proposition~\ref{zeroexistextr} in the hyperplane $\Gamma_{i,j}$  there is an index $\alpha$   covered with zero weights by all hyperplanes of other directions. So the index $\alpha$ is covered by $\Lambda$ with weight  strictly between $1 - \delta$ and $1$, that is not possible by Proposition~\ref{extrnescond}. 
\end{proof}

The last corollary of the second demand is that optimal hyperplane covers cannot have many entries that are greater than $\nicefrac{1}{2}$.

\begin{utv} \label{biglambdabound}
If $\Lambda$ is an optimal hyperplane cover of a multidimensional extremal matrix of order $n$ then the number of entries $\lambda_{i,j} > \nicefrac{1}{2}$ is less than $n$.
\end{utv}

\begin{proof}
We will say that an entry $\lambda_{i,j}$ of the optimal hyperplane cover $\Lambda$ is big if it is greater than $\nicefrac{1}{2}$ and is small otherwise.

Assume that $\Lambda$ has at least $n$ big entries. For sufficiently small $\varepsilon > 0$, consider the table $\Lambda'$ constructed from $\Lambda$ in the following way: we reduce by $\varepsilon$ all big entries and increase by $\frac{\varepsilon}{d-1}$ all small entries.

Note that $\Lambda'$ is also a hyperplane cover of the matrix $A$. Indeed,
\begin{itemize}
\item if some index was covered by $d$ small entries of $\Lambda$, then  it is covered with a bigger weight by $\Lambda'$;
\item if some index was covered by one big entry and $d-1$ small entries of $\Lambda$, then in it is covered with the same weight by $\Lambda'$;
\item if some unity entry of $A$ was covered by at least two big entries of $\Lambda$, then for sufficiently small $\varepsilon$ it is covered with weight at least $1$ by $\Lambda'$.
\end{itemize} 
By the construction, the weight of $\Lambda'$ is not greater than the weight of $\Lambda$.  Weights of $\Lambda$ and $\Lambda'$ cannot be the same because the table $\Lambda'$ cannot be an optimal hyperplane cover of $A$ (by Proposition~\ref{zeroexistextr}, any optimal hyperplane cover of an extremal matrix should contain zeroes).  If $\Lambda'$ has a smaller weight than $\Lambda$, then it violates the optimality of $\Lambda$. Obtained contradictions complete the proof.
\end{proof}

Using these propositions, we prove that Conjectures~\ref{defcon} and~\ref{multicon} are true for extremal matrices of big deficiencies.

\begin{teorema} [supporting Conjectures~\ref{defcon} and~\ref{multicon}] \label{bigdeltalambda}
Let $\Lambda$ be an optimal hyperplane cover of an extremal matrix with deficiency $\delta$.
\begin{enumerate}
\item If $\delta = 1$  then all $\lambda_{i,j}  \in \{ 0,1\}$.
\item There are no extremal matrices with deficiency $\nicefrac{1}{2} < \delta < 1$. If $\delta = \nicefrac{1}{2}$  then all $\lambda_{i,j}  \in \{ 0, \nicefrac{1}{2}, 1\}$.
\item There are no extremal matrices with deficiency $\nicefrac{1}{3} < \delta < \nicefrac{1}{2}$. If $\delta = \nicefrac{1}{3}$  then all $\lambda_{i,j}  \in \{ 0, \nicefrac{1}{3}, \nicefrac{2}{3}, 1\}$.
\end{enumerate}
\end{teorema}

\begin{proof}
Clauses 1 and 2 immediately follow from Proposition~\ref{edgebound}.

Let us prove Clause 3.  Suppose that $A$ is an extremal matrix of deficiency $\delta$ such that $\nicefrac{1}{3} \leq \delta < \nicefrac{1}{2}$ and let $\Lambda$ be an optimal hyperplane cover of the matrix $A$.  Taking into account Theorem~\ref{onereduction} that will be proved later,  we may assume that the table $\Lambda$ does not contain unity entries. 

Let $\lambda = \lambda_{i,j}$ be some nonzero entry in the table $\Lambda$ that is less than $\nicefrac{1}{2}$. Such $\lambda_{i,j}$ exists because $\Lambda$ should contain entries different from $0, \nicefrac{1}{2}$ and $1$ and because Proposition~\ref{weightone} implies that if there is some entry of $\Lambda$ greater than $\nicefrac{1}{2}$ then there exists an entry less than $\nicefrac{1}{2}$. By Proposition~\ref{edgebound}, $\lambda \geq \delta \geq \nicefrac{1}{3}$.

Assume that $\lambda > \nicefrac{1}{3}$. Then Proposition~\ref{extrnescond}  imply that all entries of $\Lambda$ less than $\nicefrac{1}{2}$ are located only within the $i$-th row.  By Proposition~\ref{boundsinrow}, all nonzero entries in this row are the same and equal to $\lambda$. Using Proposition~\ref{extrnescond} again, we deduce that all nonzero entries in other rows are also the same and equal to $1 - \lambda$ (where $\nicefrac{1}{2} < 1- \lambda <\nicefrac{2}{3}$). Note that the $i$-th row of $\Lambda$ has no more than $n-1$ $\lambda$-entries and, by Proposition~\ref{biglambdabound}, there are no more than $n-1$ $(1-\lambda)$-entries in all other rows. Then the weight of the hyperplane cover $\Lambda$ is not greater than $ n -1 < n - \delta$; a contradiction.

Therefore there are no extremal matrices whose deficiencies are strictly between $\nicefrac{1}{3}$ and $\nicefrac{1}{2}$ and for all optimal hyperplane covers $\Lambda$ of extremal matrices of deficiency $\nicefrac{1}{3}$ it holds $\lambda_{i,j} \in \{0, \nicefrac{1}{3}, \nicefrac{2}{3}, 1 \}$.
\end{proof}

At last we note that, since Conjecture~\ref{multicon} implies Conjecture~\ref{uniquecon},  extremal matrices with deficiencies $1$, $\nicefrac{1}{2}$ and $\nicefrac{1}{3}$ have unique optimal hyperplane covers.

\section{Constructions of extremal matrices}

In this section, we consider several constructions of an infinite series of extremal matrices. All these constructions support conjectures proposed in Section~\ref{problemsection}.

\subsection{Iterative constructions}

We start with constructions based on extremal matrices of a smaller size.
The first construction allows us to easily increase the dimension or discard insignificant directions of hyperplanes of extremal matrices. 

\begin{const} \label{extrconstrdim}
Let $B$ be a $d$-dimensional $(0,1)$-matrix of order $n$ in which each hyperplane of some direction  is the same $(d-1)$-dimensional  matrix $A$.  The matrix $B$ is extremal of deficiency $\delta$ if and only if the matrix $A$ is extremal of deficiency $\delta$. 
\end{const}

\begin{proof}
Without loss of generality, assume that all hyperplanes of the first direction in the matrix $B$ are the same.
Every polyplex $K$ in the matrix $A$ can be expanded to a polyplex $K'$ of the same weight in the matrix $B$ in the following way:  to each  $k_{\alpha}$, $\alpha = (\alpha_1, \ldots, \alpha_{d-1})$, we assign $k'_{\beta} = k_{\alpha}$, where $\beta = (\alpha_1, \alpha_1, \ldots, \alpha_{d-1})$.  
Conversely, for every polyplex in $B$ there is a polyplex of the same weight in $A$: an entry $k_{\alpha}$ of a polyplex in $A$ is the sum of $k'_{\beta}$ from $B$ over all $\beta = (j, \alpha_1, \ldots, \alpha_{d-1})$, $j = 1, \ldots, n$.  Since the above correspondences work when we add some entry to the support of  $A$ or $B$, this operation simultaneously produce or do not produce a polydiagonal in both matrices.
\end{proof}

Next, we describe a series of constructions that allow us to obtain extremal matrices of a greater order from a given extremal matrix.

Let $\Lambda$ be an optimal hyperplane cover of a $d$-dimensional matrix $A$ of order $n$ and of deficiency $\delta$.  On the base of table $\Lambda$, we construct the following  three  $d \times (n+1)$-tables $\Lambda'$:
\begin{const} \label{extrconstrorder}
\begin{enumerate}
\item  For an index $\alpha \in supp(A)$ covered by $\Lambda$ with weight $1$, $\Lambda'$ is obtained from $\Lambda$ by attaching entry $\lambda_{i,\alpha_i}$ to the $i$-th row,  $i = 1, \ldots, d$.
\item $\Lambda'$ is obtained  by attaching a $1$-entry to one of rows of the table $\Lambda$ and attaching $0$-entries to other rows.
\item Assume that the $j$-th row of $\Lambda$ contains a $\delta$-entry. Then we construct a table $\Lambda'$  by attaching a $(1-\delta)$-entry to one of rows of $\Lambda$ (namely, $i$-th row for some $i \neq j$), a $\delta$-entry to the $j$-th row and attaching $0$-entries to other rows.
\end{enumerate}
\end{const}

\begin{teorema} \label{extrconstrorderproof}
If $\Lambda$ is an optimal hyperplane cover of an extremal matrix $A$ of deficiency $\delta$, then tables $\Lambda'$ obtained from $\Lambda$ via Construction~$\ref{extrconstrorder}$ are optimal hyperplane covers of the extremal matrices $B = A(\Lambda')$ of order $n+1$ and of deficiency $\delta$.
\end{teorema}

\begin{proof}
Denote by $\beta$ the index of the matrix $B$  belonging to the intersection of all new hyperplanes. In all three constructions, the submatrix $B_\beta$ is  exactly the matrix $A$ and entry $b_\beta$ in matrices $B$ equals one. 
 Each  matrix $B$ has no polydiagonals, because the weight of $\Lambda'$ is less than $n+1$. Moreover, matrices $B$ contain polyplexes of weight $n+1 - \delta$ that consist of some polyplex of weight $n-\delta$ in the matrix $A$ and  $b_\beta =1$.

To prove the theorem, it remains to show that after adding some index $\gamma$ to the support of $B$  we find a polydiagonal in the resulting matrix.  If  $\gamma$ is taken from the submatrix $B_{\beta}$, then the required statement holds because of the extremality of the matrix $A$. 

Therefore, we may assume that indices $\gamma$ and $\beta$ share the same hyperplanes in directions $i_1, \ldots, i_l$.   Let us find polydiagonals using index $\gamma$ for all three clauses of Construction~\ref{extrconstrorder}.

1.  By the construction, $\alpha$ is an index of $A$ covered by $\Lambda$ with weight $1$ and $\Lambda'$ is obtained from $\Lambda$ by doubling entries corresponding to $\alpha$. Let $\sigma$ denote the index that coincides with $\beta$ in all positions except for $ \sigma_{i_j} = \alpha_{i_j}$ for all $j=1, \ldots, l$. The construction implies that submatrices $B_{\beta}$ and $B_{\sigma}$ coincide. So the submatrix $B_{\sigma}$ is equal to $A$, index $\gamma $ belongs to $ B_{\sigma}$ and  $b_{\sigma} = 1$. Thus, if we replace entry $b_\gamma$ in the matrix $B$ by $1$, then the resulting matrix has a polydiagonal.

2. By Proposition~\ref{zeroexistextr}, there exists an index $\alpha$ covered with weight $0$ by the hyperplane cover $\Lambda$. So it is sufficient to repeat the proof of the previous clause for this index $\alpha$.

3. If index $\gamma$ does not belong to a new hyperplane of direction $i$, then we can follow proofs of the previous clauses. Otherwise, $\gamma$ is an index covered with weight $1 - \delta$ by a hyperplane of direction $i$ and with weight $0$ by hyperplanes of other directions.
The submatrix $B_\gamma$ is defined by a hyperplane cover different from $\Lambda$ only in $j$-th row, where a $\delta$-entry is utilized instead of a $0$-entry. Then the support of the matrix $B_{\gamma}$ strictly contains the support of the matrix $A$. The extremality of the matrix $A$ implies that $B_{\gamma}$ has a polydiagonal. 

\end{proof}

Let us consider which of these constructions can be reversed. Note that in the first clause of Construction~\ref{extrconstrorder} it is not necessary to use an optimal hyperplane cover $\Lambda$ of some extremal matrix to obtain an optimal hyperplane cover $\Lambda'$ of another extremal matrix $B$.

Indeed, consider the following $3$-dimensional matrix of order $5$:
$$B = \left( \begin{array}{ccccc|ccccc|ccccc|ccccc|ccccc}
1 & 1 & 1 & 1 & 1 & 1 & 1 & 1 & 1 & 1 & 1 & 1 & 1 & 1 & 1 & 1 & 1 & 1 & 1 & 0 & 1 & 1 & 0 & 0 & 0 \\ 
1 & 1 & 1 & 1 & 1 & 1 & 1 & 1 & 1 & 0 & 1 & 1 & 1 & 1 & 0 & 1 & 1 & 0 & 0 & 0 & 1 & 0 & 0 & 0 & 0 \\ 
1 & 1 & 1 & 1 & 1 & 1 & 1 & 1 & 1 & 0 & 1 & 1 & \mathbb{1} & 1 & 0 & 1 & 1 & 0 & 0 & 0 & 1 & 0 & 0 & 0 & 0 \\ 
1 & 1 & 1 & 1 & 1 & 1 & 1 & 0 & 0 & 0 & 1 & 1 & 0 & 0 & 0 & 1 & 0 & 0 & 0 & 0 & 0 & 0 & 0 & 0 & 0 \\ 
1 & 1 & 1 & 1 & 0 & 1 & 0 & 0 & 0 & 0 & 1 & 0 & 0 & 0 & 0 & 0 & 0 & 0 & 0 & 0 & 0 & 0 & 0 & 0 & 0 
\end{array} \right).
$$
The optimal hyperplane cover of $B$ is
$$
\Lambda' = \frac{1}{5} \left(\begin{array}{ccccc} 4 & 2 & \mathbb{2} & 1 & 0 \\ 3 & 2 & \mathbb{2} & 1 & 0 \\ 3 & 2 & \mathbb{1} & 1 & 0 \end{array}\right).
$$
The bold entry of the matrix $B$ is covered by the hyperplane cover $\Lambda'$ with weight $1$. Note that the table $\Lambda'$   can be constructed by attaching the bold column to the table
$$\Lambda= \frac{1}{5} \left(\begin{array}{cccc} 4 & 2 & 1 & 0 \\ 3 & 2 & 1 & 0 \\ 3 & 2 & 1 & 0 \end{array}\right).$$
It can be checked that the matrix $A(\Lambda)$ is not extremal and is contained in the following $3$-dimensional extremal matrix of order $4$ with the optimal hyperplane cover $\Lambda''$:
$$\left( \begin{array}{cccc|cccc|cccc|cccc}
1 & 1 & 1 & 1 & 1 & 1 & 1 & 1 & 1 & 1 & 1 & 0 & 1 & 1 & 0 & 0 \\ 
1 & 1 & 1 & 1 & 1 & 1 & 1 & 0 & 1 & 1 & 0 & 0 & 1 & 0 & 0 & 0 \\ 
1 & 1 & 1 & 1 & 1 & 1 & 1 & 0 & 1 & 0 & 0 & 0 & 0 & 0 & 0 & 0 \\ 
1 & 1 & 1 & 0 & 1 & 0 & 0 & 0 & 0 & 0 & 0 & 0 & 0 & 0 & 0 & 0  
\end{array} \right);
~~
\Lambda'' = \frac{1}{7} \left(\begin{array}{cccc} 5 & 3 & 1 & 0 \\ 4 & 3 & 2 & 0 \\ 4 & 3 & 2 & 0 \end{array}\right).
$$

So the first clause of Construction~\ref{extrconstrorder} is not reversible. We would like to note that the above matrix $B$ is an example of an extremal matrix in which there is an index $\alpha$ covered with weight $1$ by an optimal hyperplane cover, but the submatrix $B_{\alpha}$ is not an extremal matrix.

Let us show that the second clause of Construction~\ref{extrconstrorder} can be completed to an equivalence between extremal matrices.

\begin{teorema} \label{onereduction}
Assume that a table $\Lambda'$ is obtained from $\Lambda$ by attaching a $1$-entry to one of the rows and $0$-entries to other rows. Then $A(\Lambda')$ is an extremal matrix  of deficiency $\delta$ if and only if $A(\Lambda)$ is an extremal matrix of deficiency $\delta$. 
\end{teorema}

\begin{proof}
Sufficiency was proved in Theorem~\ref{extrconstrorderproof}, so it remains to prove necessity.
Following the previous notation, let $B$ be an extremal matrix of order $n+1$ and with an optimal hyperplane cover $\Lambda'$ and let $A$ denote the matrix $A(\Lambda)$ of order $n$.

The matrix $A$ has no polydiagonals because the weight of $\Lambda$ is less than $n$. Let us assume that the weight of an optimal polyplex in $A$ is less than $n - \delta$.  Then there exists a hyperplane cover $\Lambda''$ of the matrix $A$ with a smaller weight. Attaching to $\Lambda''$ the deleted column of zero and unity entries, we obtain some hyperplane cover of $B$ whose weight is smaller than $n+1 - \delta$ that is impossible.

Let us prove the extremality of $A$. Suppose that there exists a hyperplane cover $\Lambda'''$ such that the weight of $\Lambda'''$ is less than $n$ and $A$ is strictly contained in $A(\Lambda''')$.  If we attach to the table $\Lambda'''$ the column in which tables $\Lambda$ and $\Lambda'$ differ, then we get a hyperplane cover with weight less than $n+1$ for a matrix strictly containing the matrix $B$. Again, we have a contradiction with the extremality of $B$.
\end{proof}

The third clause of Construction~\ref{extrconstrorder} is not reversible that can be illustrated by many examples. The smallest one is the first $3$-dimensional extremal matrix of order $3$ from the list of extremal matrices in Appendix.

\subsection{Extremal matrices with two-value optimal hyperplane covers}

In this section, we consider one special class of extremal matrices, namely extremal matrices that have optimal hyperplane covers $\Lambda$ for which all $\lambda_{i,j}$ are equal to $0$ or $\lambda$. We give a complete description of such matrices and their optimal hyperplane covers and prove that all conjectures from Section~\ref{problemsection} are true for them.

We start with a characterization of hyperplane covers $\Lambda$ with entries $\lambda_{i,j} \in \{ 0, \nicefrac{1}{m}\}$ for which matrices $A(\Lambda)$ contain a polyplex of the same weight as $\Lambda$.

\begin{lemma} \label{polyplexin1m}
Given $m \in \mathbb{N}$, $m < d$, let $\Lambda$ be a hyperplane cover of a $d$-dimensional matrix $A(\Lambda)$ of order $n$ with entries $\lambda_{i,j} \in \{ 0, \nicefrac{1}{m}\}$  and let the weight $W$ of $\Lambda$ be not greater than $n$. The matrix $A(\Lambda)$ has a polyplex of weight $W$  if and only if the number of nonzero entries in each row of $\Lambda$ is not greater than $W$. 
\end{lemma}

\begin{proof}
Let  $t_i$ denote the number of nonzero entries in the $i$-th row of $\Lambda$. Then $W = \sum\limits_{i=1}^ d t_i/m$.
We look for a polyplex $K$ in the matrix $A(\Lambda)$ with  $k_{\alpha} = x_{i_1,\ldots,i_m}$  if index $\alpha$ is covered with weight $\nicefrac{1}{m}$ by hyperplanes of directions $i_j$, $j = 1, \ldots, m$, and covered with weight  $0$ by hyperplanes of other directions. All other entries of $K$ we set to be zero.

In Section~\ref{condsection}, we stated that the existence of a polyplex of weight $W$ in $A(\Lambda)$ is equivalent to the existence of a nonnegative solution of some system of linear equations. Using the introduced notation, we write this system for our hyperplane cover $\Lambda$: 

\begin{gather*} \label{system1m}
\sum\limits_{i_j = i \mbox{ for some } j} \frac{x_{i_1, \ldots, i_m}}{t_i} \prod\limits_{j=1}^m t_{i_j} \frac{\prod\limits_{l=1}^d (n-t_l)}{\prod\limits_{j=1}^m (n-t_{i_j})} = 1 \mbox{ for all } i = 1, \ldots, d;  \\
\sum\limits_{i_1, \ldots, i_m} x_{i_1, \ldots, i_m} \prod\limits_{j=1}^m t_{i_j} \frac{\prod\limits_{l=1}^d (n-t_l)}{\prod\limits_{j=1}^m (n-t_{i_j})}  = W.
\end{gather*}

Application to this system the following transformation of variables 
$$z_{i_1, \ldots, i_m} = x_{i_1, \ldots, i_m} \prod\limits_{j=1}^m t_{i_j} \frac{\prod\limits_{l=1}^d (n-t_l)}{\prod\limits_{j=1}^m (n-t_{i_j})}$$
give us the system
\begin{gather*}
\sum\limits_{i_j = i \mbox{ for some } j} z_{i_1, \ldots, i_m} = t_i \mbox{ for all } i = 1, \ldots, d;  \\\sum\limits_{i_1, \ldots, i_m} z_{i_1, \ldots, i_m} = W.
\end{gather*}
Note that for the solvability of this system in nonnegative values $z_{i_1, \ldots, i_m}$ it is necessary that each $t_i$ is not greater than $W$. Sufficiency of this condition can be shown by the induction on $d$: for a minimal $t_i$, we set all $z_{i_1, \ldots, i_m}$ in the $i$-th equation to be the same and reduce this system to a system with fewer equations.
\end{proof}

With the help of this lemma, we state that there is a one-to-one correspondence between the set of nonequivalent $d$-dimensional extremal matrices of order $n$ whose optimal hyperplane covers contain only two different entries and the set of Young diagrams with $mn-1$ cells, at most $n-1$ columns and at most $d$ rows ($d > m$).

\begin{teorema} \label{allequalconstr}
\begin{enumerate}
\item If a hyperplane cover $\Lambda$ of weight $n - \nicefrac{1}{m}$ and with entries $\lambda_{i,j} \in \{0, \nicefrac{1}{m}\}$ has at least one zero in each row,  then $\Lambda$ is the unique optimal hyperplane cover of the $d$-dimensional extremal matrix $A = A(\Lambda)$ of order $n$ and deficiency of $\delta = \nicefrac{1}{m}$. 
\item There are no extremal matrices with a two-value optimal hyperplane cover that are not given by the previous clause.
\end{enumerate}
\end{teorema}

\begin{proof}
1. By conditions of the theorem, each row of $\Lambda$ has at most $n-1$ nonzero entries that is  not greater than its weight $n - \nicefrac{1}{m}$. So Lemma~\ref{polyplexin1m} implies that $\Lambda$ is an optimal hyperplane cover of the matrix $A = A(\Lambda)$. 

Let us prove that the matrix $A$ is extremal.
Suppose $\alpha$ is an index covered by $\Lambda$ with weight less than $1$. Without loss of generality, we can assume that $\alpha$ is covered with weight $1 - \nicefrac{1}{m}$. Indeed, if $\alpha$ is covered with a smaller weight then, since nonzero entries occupy at least $m$ rows of $\Lambda$, we can easily find an index $\beta$ covered with weight $1 - \nicefrac{1}{m}$ and for which the support of the submatrix $A_\beta$ is contained in the support of the submatrix $A_\alpha$.  
The submatrix $A_\alpha$ has a hyperplane cover of weight $n-1$ that satisfies the condition of Lemma~\ref{polyplexin1m}, therefore $A_{\alpha}$ contains a polydiagonal.

The uniqueness of the optimal hyperplane cover $\Lambda$  follows from Proposition~\ref{optuniupper}.

2. Assume that  $\Lambda$ is an optimal hyperplane cover of some $d$-dimensional extremal matrix $A$ of order $n$ such that all $\lambda_{i,j} \in \{0, \lambda \}$ for some $\lambda >0$.  Let us show that $\Lambda$ satisfies all conditions of Clause $1$. 

 By Theorem~\ref{slackuse}, there exists an index covered with weight $1$ by  $\Lambda$. It yields that $\lambda = \nicefrac{1}{m}$ for some $m \in \mathbb{N}$. Proposition~\ref{zeroexistextr} claims that each row of $\Lambda$ has at least one $0$-entry. At last, if the weight of $\Lambda$ is less than $n - \nicefrac{1}{m}$ then we change one of its zero entry to $\nicefrac{1}{m}$ and obtain a hyperplane cover of some matrix that contains $A$ and does not have a polydiagonal: a contradiction to the extremality of $A$.
\end{proof}

As an application of the proposed constructions, we give a characterization of all extremal matrices of deficiencies $1$ and $\nicefrac{1}{2}$.

\begin{teorema} \label{biglambdachar}
\begin{enumerate}
\item There is a one-to-one correspondence between nonequivalent $d$-dimensional extremal matrices of order $n$ and of deficiency $1$ and the Young diagrams with $n-1$ cells and no more than $d$ rows.
\item  $\Lambda$ is the optimal hyperplane cover of a $d$-dimensional extremal matrix of order $n$ and deficiency $\nicefrac{1}{2}$ if and only if all $\lambda_{i,j} \in \{ 0, \nicefrac{1}{2}, 1 \}$, weight of $\Lambda$ is $n - \nicefrac{1}{2}$ and the number of $\nicefrac{1}{2}$-entries in each row of $\Lambda$ is less  than the number of $\nicefrac{1}{2}$-entries in the union of all other rows.  
\end{enumerate}
\end{teorema}

\begin{proof}
The first clause is an immediate corollary of Theorems~\ref{bigdeltalambda} and~\ref{allequalconstr}. For the second clause we additionally use Theorem~\ref{onereduction}.
\end{proof}

\section{To diagonal extremality of extremal matrices} \label{diagextrsection}

In this section, we consider properties and questions related to diagonal extremality and prove that all constructions in the present paper produce diagonally extremal matrices. 

By the definition, each diagonally extremal matrix is extremal.  In Section~\ref{problemsection}, we put forward a conjecture that the sets of extremal and diagonally extremal matrices coincide.

On the way to prove Conjecture~\ref{diagextrcon} it is reasonable to start with its weaker version.

\begin{con} [weakening Conjecture~\ref{diagextrcon}] \label{minorcon}
If $A$ is an extremal matrix and $a_{\alpha} = 0$ then the submatrix $A_\alpha$ contains a polydiagonal.
\end{con}

Conjectures~\ref{diagextrcon} and~\ref{minorcon} may turn out to be equivalent if one more conjecture is true.

\begin{con} [connecting Conjectures~\ref{diagextrcon} and~\ref{minorcon}] \label{diagincompcon}
Assume that $A$ has a polydiagonal and $A = A(\Lambda)$ for some hyperplane cover $\Lambda$. Then $A$ contains a diagonal.
\end{con}

Note that Conjecture~\ref{diagincompcon} claims that Problem~(1) for multidimensional matrices defined by some hyperplane cover and having a polydiagonal has not only a fractional but also integer solution. 
  
We start a proof part of this section with showing that Conjecture~\ref{minorcon} holds for some zero entries of extremal matrices.

\begin{teorema} [supporting Conjecture~\ref{minorcon}] \label{extrminor}
Suppose that $\Lambda$ is an optimal hyperplane cover of an extremal $d$-dimensional matrix $A$ of order $n$ and deficiency $\delta$. If for index $\beta$, where $a_\beta = 0$, there exists an index $\alpha$ covered by $\Lambda$ with weight $1 - \delta$ and such that $\lambda_{i,\beta_i} \leq \lambda_{i,\alpha_i}$ for all $i = 1, \ldots,d$ then the submatrix $A_\beta$ has a polydiagonal. 
\end{teorema}

\begin{proof}
Since the support of submatrix $A_\beta$ contains the support of the submatrix $A_\alpha$, if $A_{\alpha}$ has a polydiagonal then $A_{\beta}$ has a polydiagonal too. So let us prove that there is a polydiagonal in the submatrix $A_\alpha$.

Let $B$ denote the matrix obtained from $A$ by changing a zero entry $a_\alpha$ to unity. 
Since $A$ is an extremal matrix, the matrix $B$ has a polydiagonal. Let $\Lambda'$ be a hyperplane cover of the matrix $B$ that coincides with $\Lambda$ in all entries except for $\lambda_{1,\alpha_1}' = \lambda_{1,\alpha_1} + \delta$. Note that $\Lambda'$ has the weight $n$, so it is an optimal hyperplane cover of $B$.

Extremality of the matrix $A$ implies that for every polydiagonal $K$ in $B$ it holds that $k_{\alpha} >0$. Assume that for some polydiagonal $K$ in $B$ we have $k_{\alpha} < 1$. Then there is another index $\gamma$ such that $k_\gamma >0$ in hyperplane $\Gamma_{1, \alpha_1}$. By Theorem~\ref{slackuse}, index $\gamma$ is covered with weight $1$ by every optimal hyperplane cover of $B$. But the hyperplane cover $\Lambda'$ covers it with weight at least $1+ \delta$: a contradiction.
\end{proof}

Next, we prove that Constructions~\ref{extrconstrdim} and \ref{extrconstrorder} (increasing dimension and order of extremal matrices) preserve the property of diagonal extremality.

\begin{teorema} [supporting Conjecture~\ref{diagextrcon}] \label{diagextrconstr}
Let $A$ be a diagonally extremal matrix, and let $B$ be an extremal matrix obtained via Construction~$\ref{extrconstrdim}$ or~$\ref{extrconstrorder}$ from $A$. Then the matrix $B$ is also diagonally extremal.
\end{teorema}

\begin{proof}
If the matrix $B$ is obtained by Construction~\ref{extrconstrdim}, then all its hyperplanes of some direction are exactly the matrix $A$. If we assume the diagonal extremality of $A$, changing a zero entry of $B$ to unity gives a diagonal in some hyperplane of $B$. It only remains to expand this diagonal to the whole matrix $B$.

Suppose now that the matrix $B$ is given by Construction~\ref{extrconstrorder}. Returning to the proof of Theorem~\ref{extrconstrorderproof}, note that in Clause 1 and 2 each polydiagonal, arising in $B$ by adding index $\gamma$ to the support, consists of some $1$-entry $b_{\sigma}$ and a polydiagonal in the submatrix $B_{\sigma} = A$ such that $\gamma \in A$. In Clause 3  for some cases the submatrix $B_{\gamma}$ strictly contains matrix $A$. So, if such operation gives a diagonal in the matrix $A$, then a diagonal also appears in the matrix $B$.
\end{proof}

To prove that all extremal matrices with two-value optimal hyperplane covers are diagonally extremal, we use the Gale--Ryser theorem.  Given two partitions $r = (r_1,\ldots,r_n)$ and $s = (s_1,\ldots,s_m)$ of the same integer number, partition $r$ \textit{majorizes}  $s$ if $r_1 +\cdots +r_k \geq s_1 +\cdots+s_k$ for all $k$. If $s$  is a nonincreasing sequence of natural numbers, then $s^*$  denotes the sequence, where $s^*_i$ is the number of $j$ such that $s_j \geq i$. 
  
\begin{teorema}[Gale~\cite{gale.flawsnet}, Ryser~\cite{ryser.01matr}] \label{galeryser}
Let $r = (r_1,\ldots,r_n)$ and $s = (s_1,\ldots,s_m)$ be two nonincreasing sequences of nonnegative integers each summing to a common value. There exists a $(0,1)$-matrix of size $n \times m$ with row sums $r$ and column sums $s$ if and only if $s^*$ majorizes $r$.
\end{teorema} 

\begin{teorema} [supporting Conjecture~\ref{diagextrcon}] \label{allequaldiagextr}
Let $A$ be an extremal matrix defined by a hyperplane cover $\Lambda$ whose entries have only two values. Then the matrix $A$ is diagonally extremal. 
\end{teorema}

\begin{proof}
By Theorem~\ref{allequalconstr}, we may assume that all entries of $\Lambda$ are equal to $0$ or $\nicefrac{1}{m}$ and the deficiency $\delta$ is  $\nicefrac{1}{m}$. 
Let $\alpha$ be an index of $A$ covered by $\Lambda$ with weight less than $1$. As in the proof of Theorem~\ref{allequalconstr}, without loss of generality, we may suppose that $\alpha$ is covered with weight $1-\nicefrac{1}{m}$ by $\Lambda$.

Let $\Lambda_\alpha$ be a hyperplane cover of the submatrix $A_\alpha$ obtained from $\Lambda$ by deleting all hyperplanes containing index $\alpha$.   Note that the existence of a diagonal in $A_\alpha$ is equivalent to that there is rearranging entries  of the $\Lambda_\alpha$ within its rows that put $\Lambda_\alpha$ to some table with all column sums at  least $1$. By the construction, $\Lambda_\alpha$ is a $d \times (n-1)$-table with exactly $m(n-1)$ nonzero entries equal to $\nicefrac{1}{m}$.  If we denote by $r_i$ the number of nonzero entries of $\Lambda_\alpha$ in the $i$-th row and by $s_j = m$ the required number of positive entries in the $j$-column, then from Theorem~\ref{galeryser} we deduce that there exists a $(0,1)$-matrix with row sums $r$ and column sums $s$.  Thus, for every $a_\alpha =0$ the submatrix $A_\alpha$ contains a diagonal.

\end{proof}

Theorems~\ref{diagextrconstr} and~\ref{allequaldiagextr} implies that all extremal matrices of big deficiencies are diagonally extremal.

\begin{teorema} [supporting Conjecture~\ref{diagextrcon}]
Every extremal matrix of deficiency $1$ or $\nicefrac{1}{2}$ is diagonally extremal.
\end{teorema}

Let us look in more details at the problem of a diagonal extremality for a general extremal matrix $A$ defined by an optimal hyperplane cover $\Lambda$. Following the proof of Theorem~\ref{allequaldiagextr}, the matrix $A$ is diagonally extremal if and only if for each index $\alpha$ with $a_\alpha = 0$ the entries in the corresponding table $\Lambda_\alpha$ can be rearranged within their rows in such a way so that the sum of entries in each column is at least $1$. After an appropriate normalization, all entries of $\Lambda_\alpha$ may be considered to be integers.

Consequently, for effective determining the diagonal extremality by a given optimal hyperplane cover of an extremal matrix, it would be convenient to have a generalization of the Gale--Ryser theorem for integer nonnegative matrices or at least some sufficient conditions on the existence of such matrices with prescribed column sums. 

\begin{problem} [related to Conjecture~\ref{diagextrcon}] \label{diagtableprob}
What are necessary and sufficient conditions on sequences of integers $r^1, \ldots, r^m$ and $s$ for the existence of a matrix $M$ such that the $i$-th row of $M$ has exactly $r^j_i$ entries equal $j$ and the sum of entries in the $l$-th column of $M$ equals $s_l$? The similar question for matrices with column sums at least $m$?
\end{problem}

Since every nonnegative integer matrix can be considered as a biadjacency matrix of a bipartite multigraph, we may state this problem in terms of valid bipartite multigraph degree sequences.

\begin{problem} [equivalent to Problem~\ref{diagtableprob}]
Assume that we are given the sets of multiplicities of all edges for one part of vertices and degrees for another part. Does there exist a bipartite multigraph satisfying these degree and multiplicity conditions?
\end{problem}

\section{Extremal matrices of small orders}

The main aim of this section is to examine multidimensional extremal matrices of orders $2$ and $3$ in view of the questions and conjectures of Section~\ref{problemsection}. For this purpose we introduce more definitions.

Firstly, we note that the support of each diagonal in a multidimensional matrix of order $2$ consists of exactly two indices, namely indices $\alpha$ and $\overline{\alpha}$ that are said to be \textit{antipodal}. Let us call a multidimensional $(0,1)$-matrix $A$ of order $2$ \textit{antipodal} if for each pair of antipodal indices $\alpha$ and $\overline{\alpha}$ entries $a_{\alpha}$ and $a_{\overline{\alpha}}$  are different. Note that antipodal matrices are exactly the graph of self-dual boolean functions.

By the definition, every antipodal matrix does not have diagonals, and adding any index to its support produces a diagonal. So all antipodal matrices with no polydiagonals are extremal. In fact, there are no other extremal matrices of order $2$.

\begin{teorema} [supporting Conjecture~\ref{diagextrcon}] \label{order2char}
Extremal multidimensional matrices of order $2$ are exactly antipodal matrices without polydiagonals. 
\end{teorema}

\begin{proof}
Suppose that there exists an extremal but not antipodal multidimensional matrix $A$ of order $2$ and deficiency $\delta$. Then there exists a pair of antipodal indices $\alpha$ and $\overline{\alpha}$ such that $a_{\alpha} = a_{\overline{\alpha}} = 0$. 

Let $\Lambda$ be some optimal hyperplane cover of the matrix $A$.
Note that the weight $2 - \delta$ of $\Lambda$ is equal to the sum of cover weights of any two antipodal indices. By Theorem~\ref{extrcode}, each of $\alpha$ and $\overline{\alpha}$ is covered by $\Lambda$ with weight less than $1$. Then both of them are covered by $\Lambda$ with weights strictly between $1 - \delta$ and $1$: a contradiction to Proposition~\ref{edgebound}.
\end{proof}

An immediate corollary of Theorem~\ref{order2char} is that every extremal matrix of order $2$ is diagonally extremal.
We would like to note that an antipodal matrix (even having a stepped structure) may contain a polydiagonal. The smallest example is the following $6$-dimensional matrix of order $2$, where the support of the polydiagonal is bold.
$$A = \left(\begin{array}{cc} 
\begin{array}{cc|cc}
1 & 1 & 1 & 1 \\ 1 & 1 & 1 & 1 \\ \hline
1 & 1 & 1 & 1 \\ 1 & 1 & 1 & 1 \\
\end{array} &
\begin{array}{cc|cc}
1 & 1 & 1 & \mathbb{1} \\ 1 & 0 & \mathbb{1} & 0 \\ \hline
1 & \mathbb{1} & 0 & 0 \\ \mathbb{1} & 0 & 0 & 0 \\
\end{array} \\
~ & ~ \\
\begin{array}{cc|cc}
1 & 1 & 1 & 0 \\ 1 & \mathbb{1} & 0 & 0 \\ \hline
1 & 0 & \mathbb{1} & 0 \\ 0 & 0 & 0 & 0 
\end{array} &
\begin{array}{cc|cc}
0 & 0 & 0 & 0 \\ 0 & 0 & 0 & 0 \\ \hline
0 & 0 & 0 & 0 \\ 0 & 0 & 0 & 0 
\end{array} 
\end{array}\right).$$

Continuing the topic of diagonal extremality, let us prove that  Conjecture~\ref{diagincompcon} from Section~\ref{diagextrsection} is true for matrices of order $2$.

\begin{utv} [supporting Conjecture~\ref{diagincompcon}] \label{diagin2}
Let $A$ be a matrix of order $2$ such that $A = A(\Lambda)$ for some hyperplane cover $\Lambda$. If $A$ contains a polydiagonal, then it has a diagonal. 
\end{utv}

\begin{proof}
Let $\Lambda$ be a hyperplane cover of a minimum weight such that $A = A(\Lambda)$. Since $A$ has a polydiagonal, the weight of a hyperplane cover $\Lambda$ is  not less than $2$. Nonexistence of diagonals in $A$ means that each unity entry of $A$ is covered by $\Lambda$ with weight at least $1 + \varepsilon$ for some $\varepsilon > 0$. But then $\Lambda' = \frac{1}{1+\varepsilon}\Lambda$ is a  hyperplane cover of $A$  with a smaller weight and such that $A = A(\Lambda')$ which is impossible.  
\end{proof}

As a corollary, we have the following alternatives for submatrices of extremal matrices of order $3$ that may help us to prove (or disprove) their diagonal extremality. 

\begin{utv}
Let $A$ be an extremal matrix of order $3$ and let $a_{\alpha} = 0$.
\begin{enumerate}
\item If the submatrix $A_\alpha$ has a polydiagonal, then $A_{\alpha}$ contains a diagonal.
\item If the submatrix $A_\alpha$ has no polydiagonals, then $A_{\alpha}$ is an extremal matrix of order $2$. 
\end{enumerate}
\end{utv}

\begin{proof}
1. Since the submatrix $A_{\alpha}$ has a polydiagonal and $A_{\alpha} = A(\Lambda)$ for some hyperplane cover $\Lambda$, Proposition~\ref{diagin2} gives that $A_{\alpha}$ has a diagonal.

2. Since $A_{\alpha} = A(\Lambda)$ for some hyperplane cover $\Lambda$ of weight at least $2$ and $A$ has no polydiagonals, we have that the matrix $A$ is antipodal (because there are no two antipodal entries both covered with weights less than $1$).  By Theorem~\ref{order2char}, every antipodal matrix of order $2$ without polydiagonals is extremal.
\end{proof}

Let us turn back to extremal matrices of order $2$.  Besides Construction~\ref{extrconstrdim},  we can propose one more iterative construction giving extremal matrices of order $2$ and of greater dimension. For this aim, we introduce the following notation.

As it was proved in Theorem~\ref{slackuse}, every row of an optimal hyperplane cover $\Lambda$ of any extremal matrix has at least one zero entry. For $d$-dimensional extremal matrices of order $2$, it means that their optimal hyperplane covers $\Lambda$ have not greater than $d$ nonzero entries $\lambda_1, \ldots, \lambda_l$, with all of them located in different rows of $\Lambda$. We will say that $\lambda_1, \ldots, \lambda_l$ are \textit{essential weights} of the optimal hyperplane cover $\Lambda$.

\begin{const} \label{extrconst2}
Given an optimal hyperplane cover $\Lambda$ of a $d$-dimensional matrix $A$ of order $2$ and of deficiency $\delta$ with essential weights $\lambda_1, \ldots, \lambda_d$, let the essential weights of the $(d+1) \times 2$-table $\Lambda'$ be $\lambda_1, \ldots, \lambda_{d-1}, \lambda_d - \delta, \delta$.
\end{const}

\begin{teorema} \label{extrconst2proof}
Let $\Lambda$ be an optimal hyperplane cover of a $d$-dimensional extremal matrix $A$ of order $2$ and of deficiency $\delta$ such that there are no entries of $A$ covered by $\Lambda$ with weight strictly between $1$ and $1+\delta$. Then the table $\Lambda'$ obtained from $\Lambda$ via Construction~$\ref{extrconst2}$ is an optimal hyperplane cover of the $(d+1)$-dimensional extremal matrix $B = A(\Lambda')$ of order $2$ and of deficiency $\delta$.
\end{teorema}

\begin{proof}
Since $B$ is defined by a hyperplane cover $\Lambda'$ of weight $2-\delta$, the matrix $B$ has no polydiagonals.

By the construction, the support of the matrix $A$ is contained in the support of the matrix $B$, more specifically, hyperplanes of the direction $d$ of the matrix $A$ are diagonally located in the intersection of the hyperplanes of directions $d$ and $d+1$ in the matrix $B$. So each polyplex $K$ of weight $2 -\delta$ in $A$ can be put to a polyplex $K'$ of the same weight in the matrix $B$.

To prove the theorem, it remains to show that the matrix $B$ is antipodal because in this case Theorem~\ref{order2char} gives that $B$ is an extremal matrix.  If the matrix $B$ is not antipodal, then there exists an index $\alpha$ covered by $\Lambda'$ with weight strictly between $1 - \delta$ and $1$. Note that by Proposition~\ref{extrnescond}, the hyperplane cover $\Lambda$ cannot cover $\alpha$ with the same weight as $\Lambda'$, because $\Lambda$ is an optimal hyperplane cover of an extremal matrix. Then by the construction, $\Lambda$ covers $\alpha$ with weight strictly between $1$ and $1 + \delta$ or with weight strictly between $1 - 2\delta$ and $1 - \delta$ but both cases contradicts to the condition of the theorem. 
\end{proof}

This construction may be used for proving Conjectures~\ref{uniquecon},~\ref{defcon} and~\ref{multicon} for matrices of order $2$.

\begin{sled} 
Assume that for every extremal matrix of order $2$ and deficiency $\delta$ there are no entries covered with weight strictly between $1$ and $1 + \delta$ by their optimal hyperplane covers. Then $\delta = \nicefrac{1}{m}$ for some $m \in \mathbb{N}$, all such extremal matrices have the unique optimal hyperplane covers $\Lambda$, with all entries $\lambda_{i,j}$ being integer multiples of $\delta$.
\end{sled}

\begin{proof}
Construction~\ref{extrconst2} allows us to obtain extremal matrices of deficiency $\delta$ with a smaller maximum essential weight from a given extremal matrix of deficiency $\delta$. By the condition, we can apply Construction~\ref{extrconst2} until all essential weights of the resulting extremal matrix are the same. Application of Theorem~\ref{allequalconstr} completes the proof.
\end{proof}

We conclude this section with one more conjecture that possibly allows us to describe all optimal hyperplane covers of extremal matrices of order $2$ in another way.

\begin{con}
Let $A$ be an extremal matrix of order $2$ and $K$ be a vertex of a polyhedron of polyplexes in $A$. Then values  $k_\alpha$, $\alpha \in supp(K)$ are essential weights of an optimal hyperplane cover for some multidimensional extremal matrix of order $2$.
\end{con}

\section*{Acknowledgements}

This work was funded by the Russian Science Foundation under grant 18-11-00136 (Sections 3--8) and is supported in part by the Young Russian Mathematics award (Sections 1--2). The author is also grateful to Vladimir N. Potapov and Sergey V. Avgustinovich for useful discussions.

\section*{Appendix. List of extremal matrices of small sizes}

We list all small-sized extremal matrices with their deficiencies $\delta$ and optimal hyperplane covers $\Lambda$ that cannot be obtained by Constructions~\ref{extrconstrdim},~\ref{extrconstrorder} and~\ref{extrconst2} from matrices of smaller dimension or order. We indicate by bold all indices covered with weight $1$ by hyperplane covers $\Lambda$. With the help of Proposition~\ref{optuniupper} it can be checked that all these extremal matrices have unique optimal hyperplane covers. Also, it can be verified directly that all these matrices are diagonally extremal.

\bigskip

\textbf{$3$-dimensional extremal matrix of order $2$}

$$ \left( \begin{array}{cc|cc}
1 & \mathbb{1} & \mathbb{1} & 0 \\ \mathbb{1} & 0 & 0 & 0 
\end{array} \right)
~~~~
\Lambda = \left(\begin{array}{cc} \nicefrac{1}{2} & 0 \\ \nicefrac{1}{2} & 0 \\ \nicefrac{1}{2} & 0 \end{array}\right) ~~~~ \delta = \nicefrac{1}{2}.
$$

\bigskip

\textbf{$4$-dimensional extremal matrix of order $2$}

$$   \left( \begin{array}{cc|cc}
1 & 1 & 1 & \mathbb{1} \\ 1 & \mathbb{1} & \mathbb{1} & 0 \\
\hline
\mathbb{1} & 0 & 0 & 0 \\ 0 & 0 & 0 & 0
\end{array} \right)
~~~~
\Lambda = \left(\begin{array}{cc} \nicefrac{2}{3} & 0 \\ \nicefrac{1}{3} & 0 \\ \nicefrac{1}{3} & 0 \\ \nicefrac{1}{3} & 0 \end{array}\right) ~~~~ \delta = \nicefrac{1}{3}.
$$

\bigskip

\textbf{$5$-dimensional extremal matrices of order $2$}

$$  \left( \begin{array}{cc} 
\begin{array}{cc|cc}
1 & 1 & 1 & 1 \\ 1 & 1 & 1 & \mathbb{1} \\
\hline
1 & 1 & 1 & \mathbb{1} \\ 1 & \mathbb{1} & \mathbb{1} & 0 
\end{array} &
\begin{array}{cc|cc}
\mathbb{1} & 0 & 0 & 0 \\ 0 & 0 & 0 & 0 \\
\hline
0 & 0 & 0 & 0 \\ 0 & 0 & 0 & 0
\end{array} 
\end{array} \right)
~~~~
\Lambda = \left(\begin{array}{cc} \nicefrac{3}{4} & 0 \\ \nicefrac{1}{4} & 0 \\ \nicefrac{1}{4} & 0 \\ \nicefrac{1}{4} & 0 \\ \nicefrac{1}{4} & 0 \end{array}\right) ~~~~ \delta = \nicefrac{1}{4}.
$$

$$  \left( \begin{array}{cc} 
\begin{array}{cc|cc}
1 & 1 & 1 & 1 \\ 1 & 1 & 1 & \mathbb{1} \\
\hline
1 & \mathbb{1} & \mathbb{1} & 0 \\ \mathbb{1} & 0 & 0 & 0 
\end{array} &
\begin{array}{cc|cc}
1 & \mathbb{1} & \mathbb{1} & 0 \\ \mathbb{1} & 0 & 0 & 0 \\
\hline
0 & 0 & 0 & 0 \\ 0 & 0 & 0 & 0
\end{array} 
\end{array} \right)
~~~~
\Lambda = \left(\begin{array}{cc} \nicefrac{1}{2} & 0 \\ \nicefrac{1}{2} & 0 \\ \nicefrac{1}{4} & 0 \\ \nicefrac{1}{4} & 0 \\ \nicefrac{1}{4} & 0 \end{array}\right) ~~~~ \delta = \nicefrac{1}{4}.
$$

$$ \left( \begin{array}{cc} 
\begin{array}{cc|cc}
1 & 1 & 1 & 1 \\ 1 & 1 & 1 & \mathbb{1} \\
\hline
1 & 1 & \mathbb{1} & 0 \\ 1 & \mathbb{1} & 0 & 0 
\end{array} &
\begin{array}{cc|cc}
1 & \mathbb{1} & 0 & 0 \\ \mathbb{1} & 0 & 0 & 0 \\
\hline
0 & 0 & 0 & 0 \\ 0 & 0 & 0 & 0
\end{array} 
\end{array} \right)
~~~~
\Lambda = \left(\begin{array}{cc} \nicefrac{3}{5} & 0 \\ \nicefrac{2}{5} & 0 \\ \nicefrac{2}{5} & 0 \\ \nicefrac{1}{5} & 0 \\ \nicefrac{1}{5} & 0 \end{array}\right) ~~~~ \delta = \nicefrac{1}{5}.
$$

\bigskip

\textbf{$3$-dimensional extremal matrices of order $3$}

$$\left( \begin{array}{ccc|ccc|ccc}
1 & 1 & 1 & 1 & 1 & \mathbb{1} & \mathbb{1} & \mathbb{1} & 0 \\ 
1 & 1 & \mathbb{1} & \mathbb{1} & \mathbb{1} & 0 & 0 & 0 & 0 \\ 
\mathbb{1} & \mathbb{1} & 0 & 0 & 0 & 0 & 0 & 0 & 0 
\end{array} \right)
~~~~
\Lambda = \left(\begin{array}{ccc} \nicefrac{2}{3} & \nicefrac{1}{3} & 0 \\ \nicefrac{2}{3} & \nicefrac{1}{3} & 0 \\ \nicefrac{1}{3} & \nicefrac{1}{3} & 0 \end{array}\right) ~~~~ \delta = \nicefrac{1}{3}.
$$
$$\left( \begin{array}{ccc|ccc|ccc}
1 & 1 & 1 & 1 & 1 & \mathbb{1} & \mathbb{1} & 0 & 0 \\ 
1 & 1 & \mathbb{1} & 1 & \mathbb{1} & 0 & 0 & 0 & 0 \\ 
1 & \mathbb{1} & 0 & \mathbb{1} & 0 & 0 & 0 & 0 & 0 
\end{array} \right)
~~~~
\Lambda = \left(\begin{array}{ccc} \nicefrac{3}{4} & \nicefrac{1}{2} & 0 \\ \nicefrac{1}{2} & \nicefrac{1}{4} & 0 \\ \nicefrac{1}{2} & \nicefrac{1}{4} & 0 \end{array}\right) ~~~~ \delta = \nicefrac{1}{4}.
$$

\bigskip

\textbf{$3$-dimensional extremal matrices of order $4$}

$$\left( \begin{array}{cccc|cccc|cccc|cccc}
1 & 1 & 1 & 1 & 1 & 1 & 1 & 1 & 1 & 1 & 1 & \mathbb{1} & \mathbb{1} & \mathbb{1} & \mathbb{1} & 0 \\ 
1 & 1 & 1 & 1 & 1 & 1 & 1 & \mathbb{1} & \mathbb{1} & \mathbb{1} & \mathbb{1} & 0 & 0 & 0 & 0 & 0 \\ 
1 & 1 & 1 & \mathbb{1} & \mathbb{1} & \mathbb{1} & \mathbb{1} & 0 & 0 & 0 & 0 & 0 & 0 & 0 & 0 & 0 \\
\mathbb{1} & \mathbb{1} & \mathbb{1} & 0 & 0 & 0 & 0 & 0 & 0 & 0 & 0 & 0 & 0 & 0 & 0 & 0  
\end{array} \right)
~~
\Lambda = \left(\begin{array}{cccc} \nicefrac{3}{4} & \nicefrac{1}{2} & \nicefrac{1}{4} & 0 \\ \nicefrac{3}{4} & \nicefrac{1}{2} & \nicefrac{1}{4} & 0 \\ \nicefrac{1}{4} & \nicefrac{1}{4} & \nicefrac{1}{4} & 0 \end{array}\right) ~~ \delta = \nicefrac{1}{4}.
$$

$$\left( \begin{array}{cccc|cccc|cccc|cccc}
1 & 1 & 1 & 1 & 1 & 1 & 1 & 1 & 1 & 1 & 1 & \mathbb{1} & 1 & \mathbb{1} & \mathbb{1} & 0 \\ 
1 & 1 & 1 & \mathbb{1} & 1 & \mathbb{1} & \mathbb{1} & 0 & \mathbb{1} & 0 & 0 & 0 & 0 & 0 & 0 & 0 \\
1 & 1 & 1 & \mathbb{1} & 1 & \mathbb{1} & \mathbb{1} & 0 & \mathbb{1} & 0 & 0 & 0 & 0 & 0 & 0 & 0 \\
1 & \mathbb{1} & \mathbb{1} & 0 & \mathbb{1} & 0 & 0 & 0 & 0 & 0 & 0 & 0 & 0 & 0 & 0 & 0  
\end{array} \right)
~~
\Lambda = \left(\begin{array}{cccc} \nicefrac{3}{4} & \nicefrac{1}{2} & \nicefrac{1}{4} & 0 \\ \nicefrac{3}{4} & \nicefrac{1}{4} & \nicefrac{1}{4} & 0 \\ \nicefrac{1}{2} & \nicefrac{1}{4} & \nicefrac{1}{4} & 0 \end{array}\right) ~~ \delta = \nicefrac{1}{4}.
$$

$$\left( \begin{array}{cccc|cccc|cccc|cccc}
1 & 1 & 1 & \mathbb{1} & 1 & 1 & 1 & \mathbb{1} & 1 & 1 & \mathbb{1} & 0 & \mathbb{1} & \mathbb{1} & 0 & 0 \\ 
1 & 1 & 1 & \mathbb{1} & 1 & 1 & 1 & \mathbb{1} & 1 & 1 & \mathbb{1} & 0 & \mathbb{1} & \mathbb{1} & 0 & 0 \\ 
1 & 1 & \mathbb{1} & 0 & 1 & 1 & \mathbb{1} & 0 & \mathbb{1} & \mathbb{1} & 0 & 0 & 0 & 0 & 0 & 0 \\ 
\mathbb{1} & \mathbb{1} & 0 & 0 & \mathbb{1} & \mathbb{1} & 0 & 0 & 0 & 0 & 0 & 0 & 0 & 0 & 0 & 0  
\end{array} \right)
~~
\Lambda = \left(\begin{array}{cccc} \nicefrac{1}{2} & \nicefrac{1}{2} & \nicefrac{1}{4} & 0 \\ \nicefrac{1}{2} & \nicefrac{1}{2} & \nicefrac{1}{4} & 0 \\ \nicefrac{1}{2} & \nicefrac{1}{2} & \nicefrac{1}{4} & 0 \end{array}\right) ~~ \delta = \nicefrac{1}{4}.
$$

$$\left( \begin{array}{cccc|cccc|cccc|cccc}
1 & 1 & 1 & 1 & 1 & 1 & 1 & 1 & 1 & 1 & 1 & \mathbb{1} & \mathbb{1} & \mathbb{1} & 0 & 0 \\ 
1 & 1 & 1 & \mathbb{1} & 1 & 1 & \mathbb{1} & 0 & \mathbb{1} & \mathbb{1} & 0 & 0 & 0 & 0 & 0 & 0 \\ 
1 & 1 & 1 & \mathbb{1} & 1 & 1 & \mathbb{1} & 0 & \mathbb{1} & \mathbb{1} & 0 & 0 & 0 & 0 & 0 & 0 \\ 
1 & 1 & \mathbb{1} & 0 & \mathbb{1} & \mathbb{1} & 0 & 0 & 0 & 0 & 0 & 0 & 0 & 0 & 0 & 0  
\end{array} \right)
~~
\Lambda = \left(\begin{array}{cccc} \nicefrac{4}{5} & \nicefrac{3}{5} & \nicefrac{2}{5} & 0 \\ \nicefrac{3}{5} & \nicefrac{1}{5} & \nicefrac{1}{5} & 0 \\ \nicefrac{2}{5} & \nicefrac{2}{5} & \nicefrac{1}{5} & 0 \end{array}\right) ~~ \delta = \nicefrac{1}{5}.
$$

$$\left( \begin{array}{cccc|cccc|cccc|cccc}
1 & 1 & 1 & 1 & 1 & 1 & 1 & 1 & 1 & 1 & 1 & \mathbb{1} & 1 & \mathbb{1} & \mathbb{1} & 0 \\ 
1 & 1 & 1 & 1 & 1 & 1 & 1 & \mathbb{1} & \mathbb{1} & 0 & 0 & 0 & 0 & 0 & 0 & 0 \\ 
1 & 1 & 1 & \mathbb{1} & 1 & \mathbb{1} & \mathbb{1} & 0 & 0 & 0 & 0 & 0 & 0 & 0 & 0 & 0 \\ 
1 & \mathbb{1} & \mathbb{1} & 0 & \mathbb{1} & 0 & 0 & 0 & 0 & 0 & 0 & 0 & 0 & 0 & 0 & 0  
\end{array} \right)
~~
\Lambda = \left(\begin{array}{cccc} \nicefrac{4}{5} & \nicefrac{3}{5} & \nicefrac{1}{5} & 0 \\ \nicefrac{4}{5} & \nicefrac{2}{5} & \nicefrac{1}{5} & 0 \\ \nicefrac{2}{5} & \nicefrac{1}{5} & \nicefrac{1}{5} & 0 \end{array}\right) ~~ \delta = \nicefrac{1}{5}.
$$

$$\left( \begin{array}{cccc|cccc|cccc|cccc}
1 & 1 & 1 & 1 & 1 & 1 & 1 & 1 & 1 & 1 & 1 & \mathbb{1} & 1 & \mathbb{1} & \mathbb{1} & 0 \\ 
1 & 1 & 1 & 1 & 1 & \mathbb{1} & \mathbb{1} & 0 & 1 & 0 & 0 & 0 & \mathbb{1} & 0 & 0 & 0 \\ 
1 & 1 & 1 & \mathbb{1} & 1 & 0 & 0 & 0 & \mathbb{1} & 0 & 0 & 0 & 0 & 0 & 0 & 0 \\ 
1 & \mathbb{1} & \mathbb{1} & 0 & \mathbb{1} & 0 & 0 & 0 & 0 & 0 & 0 & 0 & 0 & 0 & 0 & 0  
\end{array} \right)
~~
\Lambda = \left(\begin{array}{cccc} \nicefrac{4}{5} & \nicefrac{2}{5} & \nicefrac{1}{5} & 0 \\ \nicefrac{4}{5} & \nicefrac{2}{5} & \nicefrac{1}{5} & 0 \\ \nicefrac{3}{5} & \nicefrac{1}{5} & \nicefrac{1}{5} & 0 \end{array}\right) ~~ \delta = \nicefrac{1}{5}.
$$

$$\left( \begin{array}{cccc|cccc|cccc|cccc}
1 & 1 & 1 & 1 & 1 & 1 & 1 & 1 & 1 & 1 & 1 & \mathbb{1} & \mathbb{1} & \mathbb{1} & 0 & 0 \\ 
1 & 1 & 1 & \mathbb{1} & 1 & 1 & 1 & \mathbb{1} & 1 & 1 & \mathbb{1} & 0 & 0 & 0 & 0 & 0 \\ 
1 & 1 & \mathbb{1} & 0 & 1 & 1 & \mathbb{1} & 0 & \mathbb{1} & \mathbb{1} & 0 & 0 & 0 & 0 & 0 & 0 \\ 
\mathbb{1} & \mathbb{1} & 0 & 0 & \mathbb{1} & \mathbb{1} & 0 & 0 & 0 & 0 & 0 & 0 & 0 & 0 & 0 & 0  
\end{array} \right)
~~
\Lambda = \left(\begin{array}{cccc} \nicefrac{3}{5} & \nicefrac{3}{5} & \nicefrac{2}{5} & 0 \\ \nicefrac{3}{5} & \nicefrac{2}{5} & \nicefrac{1}{5} & 0 \\ \nicefrac{2}{5} & \nicefrac{2}{5} & \nicefrac{1}{5} & 0 \end{array}\right) ~~ \delta = \nicefrac{1}{5}.
$$

$$\left( \begin{array}{cccc|cccc|cccc|cccc}
1 & 1 & 1 & 1 & 1 & 1 & 1 & \mathbb{1} & 1 & 1 & 1 & \mathbb{1} & 1 & \mathbb{1} & 0 & 0 \\ 
1 & 1 & 1 & \mathbb{1} & 1 & 1 & \mathbb{1} & 0 & 1 & 1 & \mathbb{1} & 0 & \mathbb{1} & 0 & 0 & 0 \\ 
1 & 1 & \mathbb{1} & 0 & 1 & \mathbb{1} & 0 & 0 & 1 & \mathbb{1} & 0 & 0 & 0 & 0 & 0 & 0 \\ 
1 & \mathbb{1} & 0 & 0 & \mathbb{1} & 0 & 0 & 0 & \mathbb{1} & 0 & 0 & 0 & 0 & 0 & 0 & 0  
\end{array} \right)
~~
\Lambda = \left(\begin{array}{cccc} \nicefrac{3}{5} & \nicefrac{2}{5} & \nicefrac{2}{5} & 0 \\ \nicefrac{3}{5} & \nicefrac{2}{5} & \nicefrac{1}{5} & 0 \\ \nicefrac{3}{5} & \nicefrac{2}{5} & \nicefrac{1}{5} & 0 \end{array}\right) ~~ \delta = \nicefrac{1}{5}.
$$

$$\left( \begin{array}{cccc|cccc|cccc|cccc}
1 & 1 & 1 & 1 & 1 & 1 & 1 & 1 & 1 & 1 & 1 & \mathbb{1} & \mathbb{1} & 0 & 0 & 0 \\ 
1 & 1 & 1 & 1 & 1 & 1 & 1 & \mathbb{1} & 1 & \mathbb{1} & \mathbb{1} & 0 & 0 & 0 & 0 & 0 \\ 
1 & 1 & 1 & \mathbb{1} & 1 & \mathbb{1} & \mathbb{1} & 0 & 1 & 0 & 0 & 0 & 0 & 0 & 0 & 0 \\ 
1 & \mathbb{1} & \mathbb{1} & 0 & 1 & 0 & 0 & 0 & \mathbb{1} & 0 & 0 & 0 & 0 & 0 & 0 & 0  
\end{array} \right)
~~
\Lambda = \left(\begin{array}{cccc} \nicefrac{5}{6} & \nicefrac{2}{3} & \nicefrac{1}{2} & 0 \\ \nicefrac{1}{2} & \nicefrac{1}{3} & \nicefrac{1}{6} & 0 \\ \nicefrac{1}{2} & \nicefrac{1}{6} & \nicefrac{1}{6} & 0 \end{array}\right) ~~ \delta = \nicefrac{1}{6}.
$$

$$\left( \begin{array}{cccc|cccc|cccc|cccc}
1 & 1 & 1 & 1 & 1 & 1 & 1 & 1 & 1 & 1 & 1 & \mathbb{1} & \mathbb{1} & \mathbb{1} & 0 & 0 \\ 
1 & 1 & 1 & 1 & 1 & 1 & 1 & \mathbb{1} & \mathbb{1} & \mathbb{1} & 0 & 0 & 0 & 0 & 0 & 0 \\ 
1 & 1 & 1 & \mathbb{1} & 1 & 1 & \mathbb{1} & 0 & 0 & 0 & 0 & 0 & 0 & 0 & 0 & 0 \\ 
1 & 1 & \mathbb{1} & 0 & \mathbb{1} & \mathbb{1} & 0 & 0 & 0 & 0 & 0 & 0 & 0 & 0 & 0 & 0  
\end{array} \right)
~~
\Lambda = \left(\begin{array}{cccc} \nicefrac{5}{6} & \nicefrac{2}{3} & \nicefrac{1}{3} & 0 \\ \nicefrac{2}{3} & \nicefrac{1}{3} & \nicefrac{1}{6} & 0 \\ \nicefrac{1}{3} & \nicefrac{1}{3} & \nicefrac{1}{6} & 0 \end{array}\right) ~~ \delta = \nicefrac{1}{6}.
$$

$$\left( \begin{array}{cccc|cccc|cccc|cccc}
1 & 1 & 1 & 1 & 1 & 1 & 1 & 1 & 1 & 1 & \mathbb{1} & \mathbb{1} & 1 & \mathbb{1} & 0 & 0 \\ 
1 & 1 & 1 & 1 & 1 & 1 & \mathbb{1} & \mathbb{1} & 1 & 1 & 0 & 0 & \mathbb{1} & 0 & 0 & 0 \\ 
1 & 1 & \mathbb{1} & \mathbb{1} & 1 & 1 & 0 & 0 & 1 &\mathbb{1} & 0 & 0 & 0 & 0 & 0 & 0 \\ 
1 & \mathbb{1} & 0 & 0 & \mathbb{1} & 0 & 0 & 0 & 0 & 0 & 0 & 0 & 0 & 0 & 0 & 0  
\end{array} \right)
~~
\Lambda = \left(\begin{array}{cccc} \nicefrac{2}{3} & \nicefrac{1}{2} & \nicefrac{1}{3} & 0 \\ \nicefrac{2}{3} & \nicefrac{1}{2} & \nicefrac{1}{3} & 0 \\ \nicefrac{1}{2} & \nicefrac{1}{3} & 0 & 0 \end{array}\right) ~~ \delta = \nicefrac{1}{6}.
$$

$$\left( \begin{array}{cccc|cccc|cccc|cccc}
1 & 1 & 1 & 1 & 1 & 1 & 1 & 1 & 1 & 1 & 1 & \mathbb{1} & \mathbb{1} & \mathbb{1} & 0 & 0 \\ 
1 & 1 & 1 & 1 & 1 & 1 & 1 & \mathbb{1} & 1 & 1 & \mathbb{1} & 0 & 0 & 0 & 0 & 0 \\ 
1 & 1 & 1 & \mathbb{1} & 1 & 1 & \mathbb{1} & 0 & \mathbb{1} &\mathbb{1} & 0 & 0 & 0 & 0 & 0 & 0 \\ 
\mathbb{1} & \mathbb{1} & 0 & 0 & 0 & 0 & 0 & 0 & 0 & 0 & 0 & 0 & 0 & 0 & 0 & 0  
\end{array} \right)
~~
\Lambda = \left(\begin{array}{cccc} \nicefrac{2}{3} & \nicefrac{1}{2} & \nicefrac{1}{3} & 0 \\ \nicefrac{2}{3} & \nicefrac{1}{2} & \nicefrac{1}{3} & 0 \\ \nicefrac{1}{3} & \nicefrac{1}{3} & \nicefrac{1}{6} & 0 \end{array}\right) ~~ \delta = \nicefrac{1}{6}.
$$

$$\left( \begin{array}{cccc|cccc|cccc|cccc}
1 & 1 & 1 & 1 & 1 & 1 & 1 & 1 & 1 & 1 & 1 & \mathbb{1} & 1 & \mathbb{1} & \mathbb{1} & 0 \\ 
1 & 1 & 1 & \mathbb{1} & 1 & 1 & 1 & 0 & 1 & \mathbb{1} & \mathbb{1} & 0 & 0 & 0 & 0 & 0 \\ 
1 & 1 & 1 & 0 & 1 & \mathbb{1} & \mathbb{1} & 0 & \mathbb{1} & 0 & 0 & 0 & 0 & 0 & 0 & 0 \\ 
1 & \mathbb{1} & \mathbb{1} & 0 & \mathbb{1} & 0 & 0 & 0 & 0 & 0 & 0 & 0 & 0 & 0 & 0 & 0  
\end{array} \right)
~~
\Lambda = \left(\begin{array}{cccc} \nicefrac{2}{3} & \nicefrac{1}{2} & \nicefrac{1}{3} & 0 \\ \nicefrac{2}{3} & \nicefrac{1}{3} & \nicefrac{1}{6} & 0 \\ \nicefrac{1}{2} & \nicefrac{1}{3} & \nicefrac{1}{3} & 0 \end{array}\right) ~~ \delta = \nicefrac{1}{6}.
$$

$$\left( \begin{array}{cccc|cccc|cccc|cccc}
1 & 1 & 1 & 1 & 1 & 1 & 1 & 1 & 1 & 1 & 1 & \mathbb{1} & \mathbb{1} & \mathbb{1} & 0 & 0 \\ 
1 & 1 & 1 & \mathbb{1} & 1 & 1 & 1 & 0 & 1 & 1 & \mathbb{1} & 0 & 0 & 0 & 0 & 0 \\ 
1 & 1 & 1 & 0 & 1 & 1 & \mathbb{1} & 0 & \mathbb{1} & \mathbb{1} & 0 & 0 & 0 & 0 & 0 & 0 \\ 
1 & 1 & \mathbb{1} & 0 & \mathbb{1} & \mathbb{1} & 0 & 0 & 0 & 0 & 0 & 0 & 0 & 0 & 0 & 0  
\end{array} \right)
~~
\Lambda = \left(\begin{array}{cccc} \nicefrac{5}{7} & \nicefrac{4}{7} & \nicefrac{3}{7} & 0 \\ \nicefrac{4}{7} & \nicefrac{2}{7} & \nicefrac{1}{7} & 0 \\ \nicefrac{3}{7} & \nicefrac{3}{7} & \nicefrac{2}{7} & 0 \end{array}\right) ~~ \delta = \nicefrac{1}{7}.
$$

$$\left( \begin{array}{cccc|cccc|cccc|cccc}
1 & 1 & 1 & 1 & 1 & 1 & 1 & \mathbb{1} & 1 & 1 & \mathbb{1} & 0 & 1 & \mathbb{1} & 0 & 0 \\ 
1 & 1 & 1 & 1 & 1 & 1 & 1 & 0 & 1 & \mathbb{1} & 0 & 0 & \mathbb{1} & 0 & 0 & 0 \\ 
1 & 1 & 1 & \mathbb{1} & 1 & 1 & \mathbb{1} & 0 & \mathbb{1} & 0 & 0 & 0 & 0 & 0 & 0 & 0 \\ 
1 & 1 & \mathbb{1} & 0 & \mathbb{1} & 0 & 0 & 0 & 0 & 0 & 0 & 0 & 0 & 0 & 0 & 0  
\end{array} \right)
~~
\Lambda = \left(\begin{array}{cccc} \nicefrac{5}{7} & \nicefrac{3}{7} & \nicefrac{1}{7} & 0 \\ \nicefrac{4}{7} & \nicefrac{3}{7} & \nicefrac{2}{7} & 0 \\ \nicefrac{4}{7} & \nicefrac{3}{7} & \nicefrac{2}{7} & 0 \end{array}\right) ~~ \delta = \nicefrac{1}{7}.
$$

$$\left( \begin{array}{cccc|cccc|cccc|cccc}
1 & 1 & 1 & 1 & 1 & 1 & 1 & 1 & 1 & 1 & 1 & \mathbb{1} & \mathbb{1} & 0 & 0 & 0 \\ 
1 & 1 & 1 & 1 & 1 & 1 & 1 & \mathbb{1} & 1 & \mathbb{1} & 0 & 0 & 0 & 0 & 0 & 0 \\ 
1 & 1 & 1 & \mathbb{1} & 1 & 1 & \mathbb{1} & 0 & 1 & 0 & 0 & 0 & 0 & 0 & 0 & 0 \\ 
1 & 1 & \mathbb{1} & 0 & 1 & \mathbb{1} & 0 & 0 & \mathbb{1} & 0 & 0 & 0 & 0 & 0 & 0 & 0  
\end{array} \right)
~~
\Lambda = \left(\begin{array}{cccc} \nicefrac{7}{8} & \nicefrac{3}{4} & \nicefrac{1}{2} & 0 \\ \nicefrac{1}{2} & \nicefrac{1}{4} & \nicefrac{1}{8} & 0 \\ \nicefrac{1}{2} & \nicefrac{1}{4} & \nicefrac{1}{8} & 0 \end{array}\right) ~~ \delta = \nicefrac{1}{8}.
$$

$$\left( \begin{array}{cccc|cccc|cccc|cccc}
1 & 1 & 1 & 1 & 1 & 1 & 1 & 1 & 1 & 1 & 1 & \mathbb{1} & \mathbb{1} & 0 & 0 & 0 \\ 
1 & 1 & 1 & 1 & 1 & 1 & 1 & \mathbb{1} & 1 & 1 & \mathbb{1} & 0 & 0 & 0 & 0 & 0 \\ 
1 & 1 & 1 & \mathbb{1} & 1 & 1 & \mathbb{1} & 0 & 1 & \mathbb{1} & 0 & 0 & 0 & 0 & 0 & 0 \\ 
1 & \mathbb{1} & 0 & 0 & 1 & 0 & 0 & 0 & \mathbb{1} & 0 & 0 & 0 & 0 & 0 & 0 & 0  
\end{array} \right)
~~
\Lambda = \left(\begin{array}{cccc} \nicefrac{3}{4} & \nicefrac{5}{8} & \nicefrac{1}{2} & 0 \\ \nicefrac{1}{2} & \nicefrac{3}{8} & \nicefrac{1}{4} & 0 \\ \nicefrac{1}{2} & \nicefrac{1}{4} & \nicefrac{1}{8} & 0 \end{array}\right) ~~ \delta = \nicefrac{1}{8}.
$$

$$\left( \begin{array}{cccc|cccc|cccc|cccc}
1 & 1 & 1 & 1 & 1 & 1 & 1 & 1 & 1 & 1 & \mathbb{1} & 0 & 1 & \mathbb{1} & 0 & 0 \\ 
1 & 1 & 1 & 1 & 1 & 1 & 1 & \mathbb{1} & 1 & \mathbb{1} & 0 & 0 & \mathbb{1} & 0 & 0 & 0 \\ 
1 & 1 & 1 & \mathbb{1} & 1 & 1 & \mathbb{1} & 0 & 0 & 0 & 0 & 0 & 0 & 0 & 0 & 0 \\ 
1 & 1 & \mathbb{1} & 0 & \mathbb{1} & 0 & 0 & 0 & 0 & 0 & 0 & 0 & 0 & 0 & 0 & 0  
\end{array} \right)
~~
\Lambda = \left(\begin{array}{cccc} \nicefrac{3}{4} & \nicefrac{1}{2} & \nicefrac{1}{8} & 0 \\ \nicefrac{5}{8} & \nicefrac{1}{2} & \nicefrac{1}{4} & 0 \\ \nicefrac{1}{2} & \nicefrac{3}{8} & \nicefrac{1}{4} & 0 \end{array}\right) ~~ \delta = \nicefrac{1}{8}.
$$


\bib{aharoni.rysconj}{article}{
  author={Aharoni, Ron},
  title={Ryser's conjecture for tripartite 3-graphs},
  journal={Combinatorica},
  volume={21},
  date={2001},
  number={1},
  pages={1--4},
  issn={0209-9683},
}

\bib{AhaGeoSpr.pminhyper}{article}{
  author={Aharoni, Ron},
  author={Georgakopoulos, Agelos},
  author={Spr\"{u}ssel, Philipp},
  title={Perfect matchings in $r$-partite $r$-graphs},
  journal={European J. Combin.},
  volume={30},
  date={2009},
  number={1},
  pages={39--42},
  issn={0195-6698},
}

\bib{AhaHax.Hallth}{article}{
  author={Aharoni, Ron},
  author={Haxell, Penny},
  title={Hall's theorem for hypergraphs},
  journal={J. Graph Theory},
  volume={35},
  date={2000},
  number={2},
  pages={83--88},
  issn={0364-9024},
}

\bib{BestWan.Rysconj}{article}{
  author={Best, Darcy},
  author={Wanless, Ian M.},
  title={What did Ryser conjecture?},
  eprint={https://arxiv.org/abs/1801.02893},
}

\bib{FloPar.encycopt}{collection}{
  title={Encyclopedia of optimization},
  editor={Floudas, C. A.},
  editor={Pardalos, P. M.},
  note={Second edition},
  publisher={Springer, New York},
  date={2009},
  pages={xxxiii+4626},
  isbn={978-0-387-74759-0},
  isbn={978-0-387-74758-3},
}

\bib{FurKahnSeym.fracmatch}{article}{
  author={F\"{u}redi, Z.},
  author={Kahn, J.},
  author={Seymour, P. D.},
  title={On the fractional matching polytope of a hypergraph},
  journal={Combinatorica},
  volume={13},
  date={1993},
  number={2},
  pages={167--180},
  issn={0209-9683},
}

\bib{gale.flawsnet}{article}{
  author={Gale, David},
  title={A theorem on flows in networks},
  journal={Pacific J. Math.},
  volume={7},
  date={1957},
  pages={1073--1082},
  issn={0030-8730},
}

\bib{HaxNarSza.Ryserconj}{article}{
  author={Haxell, Penny E.},
  author={Narins, L.},
  author={Szab\'o, T.},
  title={Extremal hypergraphs for {R}yser's conjecture: home-base hypergraphs},
  eprint={https://arxiv.org/abs/1401.0171},
}

\bib{HaxScott.Rysconj}{article}{
  author={Haxell, Penny E.},
  author={Scott, A. D.},
  title={On Ryser's conjecture},
  journal={Electron. J. Combin.},
  volume={19},
  date={2012},
  number={1},
  pages={Paper 23, 10},
}

\bib{HuLohSud.matchhyp}{article}{
  author={Huang, Hao},
  author={Loh, Po-Shen},
  author={Sudakov, Benny},
  title={The size of a hypergraph and its matching number},
  journal={Combin. Probab. Comput.},
  volume={21},
  date={2012},
  number={3},
  pages={442--450},
  issn={0963-5483},
}

\bib{KeevMyc.geommatch}{article}{
  author={Keevash, Peter},
  author={Mycroft, Richard},
  title={A geometric theory for hypergraph matching},
  journal={Mem. Amer. Math. Soc.},
  volume={233},
  date={2015},
  number={1098},
  pages={vi+95},
  issn={0065-9266},
  isbn={978-1-4704-0965-4},
}

\bib{mangas.LPuniq}{article}{
  author={Mangasarian, O. L.},
  title={Uniqueness of solution in linear programming},
  journal={Linear Algebra Appl.},
  volume={25},
  date={1979},
  pages={151--162},
  issn={0024-3795},
}

\bib{RodlRus.hyperDirac}{article}{
  author={R\"{o}dl, Vojtech},
  author={Ruci\'{n}ski, Andrzej},
  title={Dirac-type questions for hypergraphs---a survey (or more problems for Endre to solve)},
  conference={ title={An irregular mind}, },
  book={ series={Bolyai Soc. Math. Stud.}, volume={21}, publisher={J\'{a}nos Bolyai Math. Soc., Budapest}, },
  date={2010},
  pages={561--590},
}

\bib{ryser.01matr}{article}{
  author={Ryser, H. J.},
  title={Combinatorial properties of matrices of zeros and ones},
  journal={Canadian J. Math.},
  volume={9},
  date={1957},
  pages={371--377},
  issn={0008-414X},
}

\bib{schrijv.thLP}{book}{
  author={Schrijver, Alexander},
  title={Theory of linear and integer programming},
  series={Wiley-Interscience Series in Discrete Mathematics},
  note={A Wiley-Interscience Publication},
  publisher={John Wiley \& Sons, Ltd., Chichester},
  date={1986},
  pages={xii+471},
  isbn={0-471-90854-1},
}

\bib{schrijver.combopt}{book}{
  author={Schrijver, Alexander},
  title={Combinatorial optimization. Polyhedra and efficiency. Vol. A},
  series={Algorithms and Combinatorics},
  volume={24},
  note={Paths, flows, matchings; Chapters 1--38},
  publisher={Springer-Verlag, Berlin},
  date={2003},
  pages={xxxviii+647},
  isbn={3-540-44389-4},
}

\bib{my.iter}{article}{
  author={Taranenko, Anna A.},
  title={Transversals, plexes, and multiplexes in iterated quasigroups},
  journal={Electron. J. Combin.},
  volume={25},
  date={2018},
  number={4},
  pages={Paper 4.30, 17},
}

\bib{wanless.surv}{article}{
  author={Wanless, Ian M.},
  title={Transversals in Latin squares: a survey},
  conference={ title={Surveys in combinatorics 2011}, },
  book={ series={London Math. Soc. Lecture Note Ser.}, volume={392}, publisher={Cambridge Univ. Press, Cambridge}, },
  date={2011},
  pages={403--437},
}

\bib{woodall.vecttrans}{article}{
  author={Woodall, D. R.},
  title={Vector transversals},
  journal={J. Combin. Theory Ser. B},
  volume={32},
  date={1982},
  number={2},
  pages={189--205},
  issn={0095-8956},
}



\begin{bibdiv}
    \begin{biblist}[\normalsize]
    \bibselect{biblio}
    \end{biblist}
    \end{bibdiv}
\end{document}